\documentclass[reqno]{amsart}
\usepackage{amssymb,amsmath,amsfonts,amsthm,amscd, bmpsize, bbm}
\usepackage[dvips]{graphicx}

\usepackage{psfrag}
\usepackage{graphicx, graphics, enumerate, xcolor}
\usepackage[all]{xy}
\usepackage[pagewise]{lineno} 

\theoremstyle{plain}
\newtheorem{maintheorem}{Theorem}

\newtheorem{maincorollary}[maintheorem]{Corollary}

\newtheorem{theorem}{Theorem }[section]
\newtheorem{proposition}[theorem]{Proposition}
\newtheorem{lemma}[theorem]{Lemma}

\theoremstyle{definition}
\newtheorem{remark}[theorem]{Remark}
\newtheorem*{conjecture}{Conjecture}
\newtheorem{example}[theorem]{Example}
\newtheorem{definition}[theorem]{Definition}
\newtheorem*{Theorem B}{Theorem \cite{Bothe}}

\newcommand{\R}{\ensuremath{\mathbb{R}}}

\newcommand{\diam}{\operatorname{diam}}
\newcommand{\quand}{\quad\text{and}\quad}



\begin{document}

\title[Dimension of intrinsically transversal attractors]{Dimension of a class of intrinsically transversal solenoidal attractors in high dimensions}

\author[R.~Bortolotti]{Ricardo Bortolotti }
\address{Ricardo Bortolotti - Departamento de Matemática (DMAT), Universidade Federal de Pernambuco-UFPE, Recife, Brazil}
\email{ricardo.bortolotti@ufpe.br}

\author[E.~Silva]{Eberson Ferreira da Silva }
\address{Eberson Ferreira da Silva - Departamento de Matemática, Universidade Federal Rural de Pernambuco-UFRPE, Recife, Brazil }
\email{eberson.silva@ufrpe.br}

\begin{abstract} 
We study the fractal dimension of a class of solenoidal attractors in dimensions greater or equal than 3, proving that if the contraction is sufficiently strong, the expansion is close to conformal and the attractor satisfy a geometrical condition of transversality between its components, then the Hausdorff and box-counting dimension of every stable section of the attractor have the same value, which corresponds to the zero of the topological pressure as in  Bowen's formula. {We also calculate the dimension of the attractor and prove that it is continuous in this class}.
\end{abstract}

\maketitle

\section{Introduction}
One of the most interesting problems in dynamical systems is the computation of the fractal dimension of hyperbolic sets, what provides information about the complexity of a system  originated from chaotic dynamical systems. 

This problem was originally solved in the setting of self-similar fractals by Moran \cite{moran} and for conformal repellers by Bowen \cite{bowen.conformal}.
Manning and McCluskey \cite{McCluskey-Manning} calculated the dimension of basic sets for $C^1$ Axiom A diffeomorphisms in surfaces,  proving that the dimension of the attractor is the sum of the dimension of its stable and unstable slices. This result was latter extended to conformal hyperbolic sets \cite{barreira}. In all these situations the dimensions can be obtained using thermodynamical formalism methods and corresponds to the zero of a topological pressure, this relation is called Bowen's formula.

In the non-conformal setting, Falconer \cite{falconer.affine} evaluated the dimension of generic self-affine fractals  (limit set of iterated function systems given by linear contractions)  and Simon-Solomyak \cite{simon.solomyak} extend this analysis for generic three-dimensional horseshoes. 

In general the problem of computing the dimension of a non-conformal hyperbolic set is not complete in higher dimensions.  Actually, the dimension of the hyperbolic set may not be even continuous with respect to the dynamics \cite{bdv}, so we  expect to calculate the dimension only for specific examples or under generic conditions.
For more on this problem, one can check the surveys \cite{simon.survey, barreira.gelfert}.

Further results \cite{hasselblat.schmelling, Bothe,K.Simon, reza} allowed to calculate the fractal dimension of three-dimensional solenoidal attractors that satisfy some geometrical transversality condition.

The classical solenoid attractor is given by a skew-product $T:S^{1}\times D^2\to S^{1}\times D^2 $ of the type
\begin{equation}\label{sol SW}
	T(\theta,x,y)=(2\theta \ \operatorname{mod}  1, \lambda_{1}(\theta)x+ f(\theta),\lambda_{2}(\theta)y+g(\theta)).
\end{equation}
The Smale-Williams solenoid corresponds to $\lambda_1(\theta) = \lambda_1$, $\lambda_2(\theta) = \lambda_2$, $f(\theta) = \epsilon \cos (2\pi \theta)$, $g(\theta) = \epsilon \sin (2\pi \theta)$, 
with  $0< \lambda_{1} \leq \lambda_{2} <\min\lbrace \epsilon, 1/2 \rbrace$.   
When $\Delta_T$ is conformal
($\lambda_{1}=\lambda_{2}$), it is easy to calculate  the Hausdorff dimension of the attractor $\Delta_T= \bigcap_{n\geq 0}T^{n}(S^{1}\times D^2)$, which is 
$  1+ \frac{\log 2}{-\log\lambda_2 }$ \cite{Pesin}.
%
%
When $\Delta_T$ is not conformal ($\lambda_{1}<\lambda_{2}$),  Simon \cite{K.Simon} proved that the dimension $\dim\Delta_T $ is also $1+ \frac{\log 2}{-\log\lambda_{2}}$ and  that the dimension of every stable-section $\Delta_{ T}(\theta):=\Delta_{T}\cap (\{\theta\}\times D^{2} )$  of the attractor is equals to 
$\frac{\log 2}{-\log\lambda_{2}},$
for this he used a result of Bothe \cite{Bothe} on attractors that satisfy
%
a geometric condition of transversality between its  components, called intrinsic transversality.

%
%
It is valid that \cite[Theorem A]{Bothe} if $\Delta_{T}$ is  intrinsically transversal and $\lambda_i$ is sufficiently small for $i=1,2$, then $\dim_{H}(\Delta_{T})=1+ \max\{p_{i}\}$
and 
$	\dim_{H}(\Delta_{T}\cap (\{t\}\times{D}^{2}))=p_{i}$ for all $t\in S^{1}, 
$
where  $p_{i}$ is the unique number such that the topological pressure $P(p_{i}\log \lambda_{i})$ is $0$ (Bowen's formula). 
As consequence,  it follows the dimension of the Smale-Williams solenoid attractor because it is intrinsically transversal \cite{K.Simon}.

B. Hasselblatt and J.Schmeling exhibit in \cite{hasselblat.schmelling} the following conjecture:
\begin{conjecture}\label{Conj.H-S}
	\textit{The fractal dimension of a hyperbolic set is the sum of those of its stable and unstable
		slices, where ``fractal" can mean either Hausdorff or upper box dimension.}
\end{conjecture}
They prove in \cite{hasselblat.schmelling} that this Conjecture is valid for a class of three-dimensional solenoids  with weaker assumptions about contraction than those used in \cite{Bothe}. They also show that the conjecture is still valid omitting the transversality condition and assuming that the application $T$ is analytic. However, under this weaker condition of analyticity the dimension of the stable part may not be as given by Bowen's formula (as in Example \ref{non}).

More recently, the work of \cite{hasselblat.schmelling} was extended by R. Mohammadpour, F. Przytycki and M. Rams \cite{reza}   to more general three-dimensional solenoidals attractors. Under the assumptions of intrinsic transversality and
weaker conditions of contraction, they prove that the dimensions of the cross sections are equal to the same value $t_{0}$ that corresponds to the zero of the topological pressure: $P(t_0 \log \|DT\|)=0$ (Bowen's formula). They prove that the Conjecture 
is valid for this class of attractors and, even more, they prove  results regarding the packing measure and Hausdorff measure of the stable foliations with respect to $t_{0}$.

The results mentioned above are for 3-dimensional dynamics. In higher dimensions, there are more technical difficulties for the analysis of solenoidal attractors, for example, the intersection of the strong-stable projections of two  local unstable manifolds may not be just a single point but a submanifold.
A recent result in higher dimensions was obtained in \cite{Bortolotti-Silva} for intrinsically transversal solenoidal attractors using potential-theoretic methods, 
they proved that the dimension of the attractor associated with the dynamics is the sum of the stable  and the unstable slices, but the dimensions of the cross-sections were calculated  only for almost every point in the basis.

In this work, we study a class of solenoid attractors in dimensions greater than or equal to 3, also computing the dimension of each stable section. We use some techniques used in \cite{Bothe} together with ideas presented in \cite{Bortolotti-Silva} to make a higher dimensional  analysis of the projections of the unstable manifolds. The smallest singular value will be presented as an important tool in this analysis. 

The main results of this paper states that if the contraction is sufficiently strong, the expanding map on the basis is close to conformal and the attractor is intrinsically transversal, then we can calculate the dimension of every cross-section of the attractor as the zero given by the topological pressure of the geometrical potential (Bowen's formula) and thus obtain the calculation of the dimension of the attractor. { This also allows to verify that the dimension varies continuously inside this class of attractors.}

\subsection{The Results}
Consider $V=\mathbb{T}^{l}\times E\times F$ and an embedding $T: V\to V$  of class $C^r$, $r \geq 2$,   given by 
\begin{equation}
	T(x,y,z)=(\varphi(x),\nu(x,y), \psi(x,y,z)),
\end{equation}
where $\mathbb{T}^{l}=\mathbb{R}^{l} / \mathbb{Z}^{l}$ is the $l$-dimensional torus, $E\subset\mathbb{R}^{p}$ and $F\subset\mathbb{R}^{d}$ are convex open bounded sets.
Suppose that
$\varphi:\mathbb{T}^{l}\to \mathbb{T}^{l} $ is an expanding map of degree $N\geq 2$,
$\nu:\mathbb{T}^{l}\times \mathbb{R}^{p}\to E$ and $\psi:\mathbb{T}^{l}\times \mathbb{R}^{p} \times \mathbb{R}^{d}\to F$ are $C^{r}$ applications that are both contractions and $\psi$ contacts stronger, that is,
\begin{equation}\label{eq 2.2.1}
	0<\lVert( D_{z}\psi(x,y,z)) |_{Z} \rVert<\lVert( D_{y}\nu(x,y)) |_{Y} \rVert<{1}.
\end{equation}
for   $Y=\{0\}\times\mathbb{R}^{p}$, $Z=\{0\}\times \{0\} \times\mathbb{R}^{d}$ and every  $(x,y,z)$.
We also suppose that
$D_{y}\nu(x,y)|_{Y}$ is conformal, that is, \begin{equation}\label{eq 2.1.1}
	\lVert D_{y}\nu(x,y)|_{Y}  \rVert= \lVert( D_{y}\nu(x,y)|_{Y})^{-1}  \rVert^{-1}.
\end{equation}

Let $\mathcal{T}$ be the set of all $C^{r}$ embeddings  $T:V\to V$  described as above. For every $T\in \mathcal{T}$, the corresponding attractor is the set  $\Delta_{T}=  \bigcap_{n\geq 0}\overline{T^{n}(V)}$ and the stable section in $x\in\mathbb{T}^{l}$ is the set $\Delta_{T}(x)=\Delta_{T}\cap \mathbb{D}(x)$, where $\mathbb{D}(x)= \{x\}\times E\times F$.
%
%
%
We say that $B$ is a component of $\Delta_T$ if it is a $l$-dimensional submanifold   contained in $\Delta_T$. 
Denote $\rho(x,y,z)=(x,y)$ and  $\pi(x,y,z)=x$. 

\begin{definition}\label{def Transv}

The attractor $\Delta_T$ is said \textbf{intrinsically transversal} if for every ball $B \subset \mathbb{T}^{l}$ with radius smaller than $\frac{1}{2}$ and for every two distinct components $B_{1},B_{2}$ of $\Delta_{T}\cap \pi^{-1}(B)$, it is valid that the submanifolds  $\rho(B_{1})$ and $\rho(B_{2})$ are transversal.
\end{definition}

Remind that two submanifolds $S_1, S_2 \subset M$ are said transversal if $T_{p}S_1 + T_{p}S_2 = T_{p}M$ for every $p \in S_1 \cap S_2 $. 

The geometrical condition of intrinsic transversality means that the overlaps between $\rho$-projections of distinct components of the attractor do not occupy much space. 
If the $\rho$-projections of the components of the attractor were disjoint (they are not), it could be easier to prove that the same standard upper bound for the dimension of the attractor is also a lower bound. Intrinsic transversality allows to prove that the overlaps do not affect the calculation of the dimension.

Denote also  $\lambda(x,y):=\lVert D_{y}\nu(x,y) ) |_{Y} \rVert$, $\underline{\lambda}=\inf\{\lambda(x,y)\}$, $\overline{\lambda}=\sup\{\lambda(x,y) \}$, $\overline{\beta}=\sup\lVert D\varphi(x)\rVert$  and $\underline{\beta}= \sup\lVert(D\varphi(x))^{-1}\rVert^{-1}$.    
Consider the sets
\begin{equation}\label{Eq.8}
\mathcal{T}^{*}=\Bigg\{T\in\mathcal{T}: \overline{\lambda}< N^{-max\{l,2\}} 
\text{,} \quad
\overline{\lambda}< \Bigg( \overline{\beta}^{l}\underline{\beta}^{\frac{2\log(N)}{\log(\underline{\beta}\overline{\lambda}^{-1})}- l  }\Bigg)^{\frac{2\log(\underline{\lambda})}{\log(N)}   }\Bigg\}
\end{equation} 
\begin{equation}\label{close.conformal}
\text{and } \quad \mathcal{E}^{*}=\{T\in\mathcal{T}: \overline{\beta}^{l}< N^{\frac{1}{2}}\underline{\beta}^{p} \}. 
\end{equation}

In Section 2, we will define a semi-conjugation  $h:\Sigma_A\to \Delta_T$ of $T$  with a bi-lateral shift $\sigma$, the geometric potential $\phi:\Sigma_A\to \mathbb{R}$     by
$
\phi(\underline{\tilde{a}})=\log\lVert D_{y}\nu(h(\underline{\tilde{a}}))|_{Y }\rVert
$
and $d_0>0$ so that $P(\sigma,d_{0}\phi)=0$ (Bowen's equation).
This number $d_0$ is the expected value for the dimension of every stable section, and this is what we prove for the dynamics that we are considering.

\begin{maintheorem}\label{Teo 5.1}
If $T\in \mathcal{{T}}^{*} \cap \mathcal{E}^* $ and $\Delta_T$ is intrinsically transversal, then:
\begin{equation}\label{dimension}
	\operatorname{dim}_{B}\rho(\Delta_{T}(x))= \operatorname{dim}_{H}\rho(\Delta_{T}(x))=  \dim_{B}(\Delta_{T}(x))=\dim_{H}( \Delta_{T}(x)) = d_{0} 
\end{equation} 
for all $x\in\mathbb{T}^{l}$.	

\end{maintheorem}

Above we denote $\dim_H(X)$ and $\dim_B(X)$ for the Hausdorff and box dimension of the set $X$ (see Section \ref{Sec2} for the definitions).

As consequence, it follows the dimension of the attractor.

\begin{maincorollary}
If $T\in \mathcal{{T}}^{*} \cap \mathcal{E}^* $ and $\Delta_T$ is intrinsically transversal, then: \begin{equation}\label{cor}
	\dim_{B}(\Delta_{T})=\operatorname{dim}_{H}( \Delta_{T})= \ \operatorname{dim}_{B}(\rho(\Delta_T)) = \operatorname{dim}_{H}(\rho(\Delta_T)) =l+d_{0}.\end{equation}
\end{maincorollary}

The value $l+d_0$ for the dimension of the attractor $\Delta_T$ is a consequence of the value $d_0$ for $\Delta_T(x)$ for every $x \in \mathbb{T}^l$. It 
follows from an auxiliary result (Theorem \ref{Teo5.1}) that we will prove in Section 4.

{Denoting $d_0= d_0(T)$, we can notice that this value is continuous with respect to $T$. For this, let $\mathcal{U}$ be the set of the mappings $T=T(\varphi, \nu, \psi)$ satisfying \eqref{eq 2.2.1}, \eqref{eq 2.1.1}, 
$T\in \mathcal{{T}}^{*} \cap \mathcal{E}^* $ and such that $\Delta_T$ is intrinsically transversal.}
\begin{maincorollary}
{	The mapping $d_0: \mathcal{U} \to \R$ is continuous. So it follows that the Hausdorff and box dimensions of the attractor and its stable slices 		
	are continuous with respect to $T$.}

{	Moreover, if $p=1$ then $\mathcal{U}$ is $C^1$-open.}
\end{maincorollary}

The condition
$T \in \mathcal{T}^{*}$ means that the contraction is sufficiently strong and $T \in \mathcal{E}^*$ means that the expanding map is close to conformal. To illustrate these conditions, let us give three class of examples.

\begin{example}
Consider $T:\mathbb{T}^l \times \R^p \times \R^d \to \mathbb{T}^l \times \R^p \times \R^d $ given by $$T(x,y) = (E(x), C(y) + f(x), \psi(x,y,z)  ), $$ where $l\leq p < 2l$, $E$ is an  expanding linear endomorphism that is multiple of the identity, $C$ is a conformal linear isomorphism of $\R^p$ with $\lambda=\|C\|$ sufficiently small so that \eqref{Eq.8} is valid for $\lambda=\overline{\lambda}=\underline{\lambda}$ and $\psi(x,y,\cdot)$ is another contraction with $\| D_z\psi \| < \lambda$. 
So $T\in \mathcal{{T}}^{*} \cap \mathcal{E}^* $. 
\end{example}

\begin{example}\label{non}
Consider $T:\mathbb{T}^l \times \R^p \to \mathbb{T}^l \times \R^p $ given by $$T(x,y) = (\varphi(x), C(y) + f(\varphi(x)) - C(f(x)),$$ where $\varphi$ is an expanding map, $C(x,\dot)$ is a contraction and  $f: \mathbb{T}^l \to \R^p$ is analytic (or just Lipschitz). It follows  that $\Delta_T = \operatorname{graph}(f)$, so its dimension is $l$ and each $\Delta_T(x)$ consists in only one point, then $\dim_H(\Delta_T(x))=0$. In this case we have that $\Delta_T$ is not intrinsically tranversal.

This example shows that the expected dimension given by Bowen's formula is not valid in general, although the Conjecture of Hasselblat-Schmelling is still valid in this case, since $\dim_H \Delta_T = l = l + \dim_H\Delta_T(x)$.
\end{example}


The sets $\mathcal{T}^*$ and  $\mathcal{E}^*$ are open in the $C^1$-topology.  The property of intrinsic transversality is also  $C^1$-open \cite{Bortolotti-Silva}.   The product of intrinsically transversal attractors is also intrinsically transversal \cite{Bortolotti-Silva}.
Intrinsic transversality is common for attractors with strong contraction: in dimension 3 it is $C^1$-generic when $\overline{\lambda} < N^{-2}$ \cite[Theorem B]{Bothe}.  We believe that it is also $C^1$-generic in any dimension under the condition of $\overline{\lambda} $ small enough.

{Putting togheter these results we  give an example as follows.}

\begin{example}
Consider three-dimensional solenoidal attractors given by mappings
\begin{align*}
	T_i(x,y,z) &= (\mu_i x, \lambda y + a_i(x), \tilde \lambda_i z + b_i(x) ) \quad \quad i=1, \cdots, k
\end{align*}
that are intrinsically transversal attractors, where $\mu_i \geq 2$ are integers, $\lambda \in (0,1)$ and $ \tilde \lambda_i <  \lambda  $ for every $i=1,\cdots,k$.

We have that  $T \in \mathcal{E}^*$ if $\max \mu_i^k < (\mu_1\cdots \mu_k)^{1/2} \min \mu_i^k$ (which is valid if $\mu_1=\cdots =\mu_k$), we have that $T \in \mathcal{T}^*$ if we  consider $\lambda$ sufficiently small and we can notice that there exists such functions $a_i$ so that the attractors are intrinsically transversals by \cite[Theorem B]{Bothe}.

So $T= T_1 \times \cdots \times  T_k$ and its perturbations in $\mathcal{U}_T$ are mappings satisfying the hypothesis of Theorem A and Corollary B ($T\in \mathcal{{T}}^{*} \cap \mathcal{E}^* $ and $\Delta_T$ is intrinsically transversal).	
\end{example}

{This paper is organized as follows. In Section 2 we give the definitions of Hausdorff and box dimensions, topological pressure, the codification of the dynamic and the expected value $d_0$ for the dimension. There we also define the smallest singular value, which will allow to analyze the structure of the attractor close to transversal intersections of its components.
In Section 3 we prove the Geometric Lemma, which is the main step in the proof of Theorem A.
In Section 4 we consider subsets of the attractor for which we can calculate the lower bound for its dimension and we analyze the unstable holonomies for points in this subset, proving that they are bi-Lipschitz, this will allow to compute the dimension for every stable section and of the attractor. Finally we notice that the dimension is continuous in this class of solenoids.}



\section{Preliminaries}\label{Sec2}
\subsection{Box counting dimension and Hausdorff dimension}
The notions of fractal dimension that we will use in this work will be box counting dimension and Hausdorff dimension. We will give their definitions and suggest reference \cite{Falconer} for more details.

Consider $X$ a metric space. The \textit{lower} and \textit{upper box-counting dimensions} of $M\subset X$ are, respectively \begin{equation*}
\underline{\operatorname{dim}}_{B}M:=\displaystyle\liminf_{\epsilon\to 0}\dfrac{\log N(M,\epsilon)}{-\log\epsilon}
\quand	\overline{\operatorname{dim}}_{B}M:=\displaystyle\limsup_{\epsilon\to 0}\dfrac{\log N(M,\epsilon)}{-\log\epsilon}.
\end{equation*}
where $N(M,\epsilon)$ denotes the smallest number of balls of radius $\epsilon$ that are needed to cover the set $M$. When $\underline{\operatorname{dim}}_{B}M=\overline{\operatorname{dim}}_{B}M$, the \textbf{box-counting dimension} of $M$ is the limit 
\begin{equation*}
{\operatorname{dim}}_{B}M=\displaystyle\lim_{\epsilon\to 0}\dfrac{\log N(M,\epsilon)}{-\log\epsilon}.
\end{equation*}

The diameter of  $M\subset X$ is given by
$	\diam M=\sup\{\operatorname{d}(p,q); p,q\in M\}$
and the diameter of a collection $\mathcal{M}$ of subsets of $X$ by \begin{equation*}
\diam\mathcal{M}=\sup\{\diam M: M\in \mathcal{M}\}.
\end{equation*} Given $t\geq 0$
{ and a cover $\mathcal{M}$ of $M$},
we define
\begin{equation}
H_t(\mathcal{M}):= \sum_{M\in \mathcal{M}}(\diam M)^{t}
\end{equation}
The $t$\textbf{-dimensional Hausdorff measure} of $M$ is defined  by \begin{equation*}
\mathcal{H}^{t}(M)=\lim_{\epsilon\to 0}\mathcal{H}_{\epsilon}^{t}(M),
\end{equation*} where $$\mathcal{H}^{t}_{\epsilon}(M)=\inf_{\mathcal{M}} {H_t(\mathcal{M})}$$ and the infimumm is taken over all  collections $\mathcal{M}$ that cover the set $M$ and such that  $\diam\mathcal{M}\leq \epsilon$. 

There exists a number $d$ such that $\mathcal{H}^{t}(M)=\infty$ for all $t<d$ and $\mathcal{H}^{t}(M)=0$ for all $t>d$, this number $d$ is the \textbf{Hausdorff dimension} of $M$:
$$\operatorname{dim}_{H}(M)=\inf\{t\geq 0 ; \mathcal{H}^{t}(M)=0 \}=\sup\{t\geq 0 ;\mathcal{H}^{t}(M)=\infty\}.$$
It is possible to show that the following inequality holds \begin{equation}\label{eq.B-H}
\operatorname{dim}_{H}X\leq \underline{\dim}_{B}X\leq \overline{\dim}_{B}X.
\end{equation} In general these inequalities in (\ref{eq.B-H}) can be strict. Equality only occurs in some specific situations.

\subsection{Codification of the dynamics}

First we consider a codification for the expanding map $\varphi$.
Consider $\mathcal{R}=\{ R_1, R_2, \dots , R_{s}\}$ a Markov partition of $\mathbb{T}^{l}$ with respect to $\varphi$ whose diameter $\gamma=\operatorname{diam}(\mathcal{R})$ satisfies  $0<\gamma< \min\{ \frac{1}{2},\alpha\}$, where $\alpha$ is an expansivity constant of $\varphi$ and such that the contractive inverse branches $\varphi^{-1}:\mathbf{B}(x,\gamma)\to \mathbb{T}^{l}$ are well defined in balls of radius $\gamma$.

Consider $\hat{I}=\{1,2,...,s\}$ and $\hat{I}^{n}$ the set of words with letter in $I$ and length  $n$, $1 \leq n \leq \infty$. Define the subset of admissible words $I_{n}:=\{\underline{a}=(a_1,a_2,...,a_n)\in \hat{I}^{n}, \varphi(R_{a_{j}})\cap R_{a_{j+1}}\neq\emptyset \ \forall j=1,...,n-1 \}$.

Now we consider a codification for $T$. 	
For each $\underline{a}=(a_1,...,a_n)\in I_n$, the set $\mathbf{T}_{\underline{a}}^{'}=\cap_{j\geq 0}^{n-1}\varphi^{-j}(R_{a_{j+1}})$  is non-empty if and only if $\underline{a}\in I_n $. We  also denote $\mathbf{T}_{\underline{a}}=\overline{\mathbf{T}_{\underline{a}}^{'}}$.
\begin{definition}
Given $B\subset\mathbb{T}^{l}$, the \textbf{components of $T^{k}(V)\cap \pi^{-1}(B)$} are the sets $T^l(\mathbf{T}_{\underline{a}}\times E \times F) \cap \pi^{-1}(B)$, for  $\underline{a} \in I_k$  such that  $\varphi^{k}(\mathbf{T}_{\underline{a}})\cap B\neq\emptyset$. 
\end{definition}


Consider the mapping $
\pi_{m}:\cup_{n\geq m}I_n \to  I_m$, $ 
\pi_{m}(a_1,a_2,...)= (a_1,a_2,...,a_{m})$.
%
Define also $\tau:I_{\infty}\longrightarrow \mathbb{T}^{l}$ by $\tau(a_1,a_2,...)=\cap_{j\geq 0}\varphi^{-j}(R_{a_{j+1}}), $
which is injective and conjugates $\varphi$ with the {one-sided shift $\sigma^+$} in $I_{\infty}$.

For fixed $n$, the family $\mathcal{T}_{n} := \{ \mathbf{T}_{\underline{a}},   \underline{a}\in I_n\}$ is another Markov partition of $\mathbb{T}^{l}$ and  
the diameter of each element of $\mathcal{T}_{n}$ satisfies $K_{0}\overline{\beta}^{(1-n)}\leq \diam \mathbf{T}_{\underline{a}}^{'} \leq \underline{\beta}^{(1-n)}\gamma,$ where $K_{0}=\min_{1\leq j\leq s}\diam(R_{j})$. 



For each $x\in \mathbb{T}^{l}$, fix a letter  $s(x) $  in $I$ such that $x\in{\overline{R}}_{s(x)}$, 
define the sets
$I^{n}(x):=\{\underline{a}\in I_n:\varphi^{n}(\mathbf{T}_{\underline{a}})\cap {\overline{R}}_{s(x)}\neq \emptyset \}$
and  $I^{\infty}(x)=\{\underline{a}\in I_{\infty}:[\underline{a}]_{n}\in I^{i}(x), \  \forall n \geq 1 \}$. 
Given $\underline{a}\in I^{n}(x)$, denote $\underline{a}(x) $ the point $x_0$ of $\mathbf{T}_{\underline{a}}$ such that  $\varphi^{n}(x_0)=x$, denote also  $[\underline{a}]_{j}$ the truncation of $\underline{a}$ of length $j$.
For $\underline{c}\in I_m$, denote $I^{n}(\underline{c})=\{\underline{a}\in I_n:\varphi^{n}(\mathbf{T}_{\underline{a}})\cap \mathbf{T}_{\underline{c}}\neq \emptyset \}$, $	I^{\infty}(\underline{c})=\{\underline{a}\in I_{\infty}: [\underline{a}]_{n}\in I^{n}(\underline{c}), \ \forall \ n\geq 1 \}$ {and $\tilde M:= \bigcup_{\underline{a} \in I_1} \mathbb{D}(\underline{a}) \times I^{\infty}(\underline{a})$ }.

Given $\underline{a} \in I_n$, the image of $\mathbf{T}_{\underline{a}}\times \{0\}\times\{0\}$ by $T^{n}$ is the graph of the application  $S(\cdot,\underline{a}):\mathbb{D}(\underline{a})\subset\mathbb{T}^{l}\to \mathbb{R}^{p}\times\mathbb{R}^{d}$ given by
\begin{equation}
S(x,\underline{a})=\big(\nu_{[\underline{a}]_{1}(x)} \circ\nu_{[\underline{a}]_{2}(x)}\circ...\circ\nu_{\underline{a}(x)}(0) , \psi_{[\underline{a}]_{1}(x)} \circ\psi_{[\underline{a}]_{2}(x)}\circ...\circ\psi_{\underline{a}(x)}(0,0) \big),
\end{equation}
where $\psi(x,y)=\psi_{x}(y)$,  $\nu(x,z)=\nu_{x}(z)$ and  $\mathbb{D}(\underline{a})=\{ x\in\mathbb{T}^{l}: \underline{a}\in I^{n}(x) \}$.
Since the inverse branches are well defined in balls of radius $\gamma$ for every $x$, we can extend $S(.,\underline{a})$ to a function in $B(x,\gamma)$.

For $\underline{a}\in I^\infty(x)$, the sequence $(S(x,[\underline{a}]_{n}))_{n\geq 1}$ converges uniformly to a function $S(x,\underline{a})$ and the sequence of the  derivatives of order $j$, $1\leq j \leq r$,   $(D^jS(x,[\underline{a}]_{n}))_{n\geq 1}$ converges uniformly to $D^jS(x,\underline{a})$.
Also, there exists a constant $\kappa >0$ such that $\lVert D^j S(x,\underline{a})\rVert \leq \kappa$ for every $\underline{a} \in I_n$ and $x \in \mathbb{D}(\underline{a})$.


\begin{remark}
For $\underline{a}\in I^{n}(x)$,  we have the expression 
$$\rho DS(x,\underline{a})=  D_{x}\nu(x_{1},y_{1})D\varphi(x_{1}))^{-1} +  \sum_{i=1}^{n}(\prod_{j=2}^{i-1}D_{y}\nu(x_{j},y_{j}))D_{x}\nu(x_{i},y_{i})D\varphi^{i}(x_{i}))^{-1},
$$
where $x_{i}=[\underline{a}]_{i}(x), \ y_{i}=\nu_{[\underline{a}]_{i+1}(x)} \circ...\circ\nu_{\underline{a}(x)}(0)$ and $y_{n}=0$.
\end{remark}

\begin{remark} The graphs of $S(.,\underline{a})$ correspond to the components of the attractor, which are unstable manifolds. More precisely:

\begin{enumerate}
	\item  If $(x,y,z)\in\Delta_{T},$ then $(x,y,z)=(x,S(x,\underline{b}))$ for some $\underline{b}\in I^{\infty}(x)$.
	
	\item   If $(x,y,z)\in\Delta_{T}$ and $\underline{a}\in I_n$ is such that $(x,y,z)\in Z_{\underline{a}}$, where  $Z_{\underline{a}}$ is the component described by $\underline{a}$, then there exists  $a_{\infty}\in {I}^{\infty}(x)$, with $\pi_{n}(a_{\infty})=\underline{a}$, and for $n\geq 0$ if we write $x_{-n}=[a_{\infty}]_{n}(x)$ then  $E^{u}_{(x_{-n},y_{-n},z_{-n})}=\operatorname{graf}(DS(x_{-n},\sigma^{n}(a_{\infty})))$. 
\end{enumerate}
\end{remark}


Define the matrix  $A=[a_{ij}]_{s\times s}$ such that $a_{ij}=1$ if $\varphi(R_{i})\cap R_{j}\neq \emptyset$ and $a_{ij}=0,$ otherwise. Denote  $\hat{I}^{\mathbb{Z}}=\{1,...,s\}^{\mathbb{Z}}$ the set of bilateral sequences $\underline{\tilde{a}}=(...,\tilde{a}_{-1},\tilde{a}_{0},\tilde{a}_{1},...)$,  $\tilde{a}_{k}\in \{1,2,...,s\}$.
Consider  $\Sigma_A$ the subset of  $\hat{I}^{\mathbb{Z}}$ of admissible words, that is, 
$\Sigma_A=\{(...,\tilde{a}_{-1},\tilde{a}_{0},\tilde{a}_{1},...)\in \hat{I}^{\mathbb{Z}}: a_{\tilde{a}_{n},\tilde{a}_{n+1}}=1 \  \mbox{ for every } \ n\in\mathbb{Z}\}.$

Let us  see $\Sigma_A$ as a subset of $ I_{-\infty}\times I_{\infty}$, where $I_{-\infty}=\{(...,\tilde{a}_{-2},\tilde{a}_{-1},\tilde{a}_{0}): a_{\tilde{a}_{-n-1},\tilde{a}_{-n}}=1\  \mbox{for all } \ n\in\mathbb{N}\cup\{0\} \},$
and we write each element $\underline{\tilde{a}} \in \Sigma_A$ as a pair  $\underline{\tilde{a}}=(\underline{\tilde{a}}^-,\underline{{a}})$, where $\underline{\tilde{a}}^-=(...,\tilde{a}_{-2},\tilde{a}_{-1},\tilde{a}_{0})$
and $\underline{{a}}=(\tilde{a}_{1},\tilde{a}_{2},...)$.

Define the shift  $\sigma:\Sigma_A\to \Sigma_A$ given by $\sigma((\tilde{a}_{n})_{n\in\mathbb{Z}})=(\tilde{a}_{n+1})_{n\in\mathbb{Z}}$ and the semi-conjugation 
{ $h:\Sigma_A\to \Delta_T$ by 
	$h(\underline{\tilde{a}})=(x,S(x,\underline{a})),$ where $x$ is such that $\varphi^{i}(x)\in R_{{\tilde{a}}_{-i}}$ for every $i\geq 0$ {(i.e., $x= \tau \circ R(\underline{\tilde{a}}^-)$)}. 
	We have that $h$ is surjective and $h\circ {\sigma^{-1} }=T\circ  h$.   
	
	\begin{remark}
		{Above it appeared $\sigma^{-1}$ because in this notation the word $\underline{\tilde{a}}^-$ corresponds to the reversed itinerary of $x$. Considering $R: I_{-\infty} \to I_\infty$ defined by $R(\ldots, \tilde{a}_{-1}, \tilde{a}_{0}) = (\tilde{a}_{0}, \tilde{a}_{1}, \ldots)$ and $\sigma^-(\ldots, \tilde{a}_{-1}, \tilde{a}_{0})=(\ldots, \tilde{a}_{-2}, \tilde{a}_{-1})$ the one-sided shift in $I_{-\infty}$, we have that $\varphi \circ (\tau \circ R) = (\tau \circ R) \circ \sigma^-$.}
	\end{remark}

	\subsection{Topological pressure}
	Let $X$ be a compact metric space and let $f:X\to X$ be a continuous transformation. For each $n\in\mathbb{N}$ we define \begin{equation*}
		\operatorname{d}_{n}(p,q)=\max\{\operatorname{d}(f^{k}(p),f^{k}(q));0\leq k\leq n-1\}
	\end{equation*}Given $\epsilon>0$, a finite set $M\subset X$ is said to be $(n,\epsilon)-$separated if $\operatorname{d}_{n}(p,q)>\epsilon$ for every $p,q\in M$ with $p\neq q$.  
	
	The \textit{topological pressure} of a continuous function $\phi:X\to\mathbb{R}$ (called \textbf{potential}) with respect to  $f$ is defined by
	
	\begin{equation*}
		P(f,\phi)=\lim_{\epsilon\to 0}\limsup_{n\to +\infty}\frac{1}{n}\log\sup_{M}\sum_{p\in M}\exp\sum_{k=0}^{n-1}\phi(f^{k}(p)),
	\end{equation*}
	where the supremum is taken over all $(n,\epsilon)-$separated sets $M\subset X$.

	Consider the geometric potential  geometric potential $\phi:\Sigma_A\to \mathbb{R}$    given  by
	$
	\phi(\underline{\tilde{a}})=\log\lVert D_{y}\nu(h(\underline{\tilde{a}}))|_{Y }\rVert
	$. The topological pressure of $s\phi$ with respect to $\sigma$ is 
	\begin{equation}\label{pressure}
		P(\sigma, s\phi)=\lim_{n\to\infty}\dfrac{1}{n}\log\big(\sum_{\underline{a}\in I_{n}}\exp\big(s\sum_{i=0}^{n-1}\phi(\sigma^{i}(\underline{a}))\big)\big).
	\end{equation}

	Since $P(\sigma, s\phi)$ is decreasing with respect to $s$, $P(s, 0)=\log N$ and $\underset{s\to + \infty}{\lim}P(\sigma, s\phi)= -\infty$, so  there exists a unique $d_{0}>0$ such that $P(\sigma,d_{0}\phi)=0$}

	This number $d_0$ will allow to compute the dimension of the attractor and its stable sections. 
	The upper bounds for the dimensions follows in a standard way:
	%
	\begin{proposition}[Theorem 3.3 in \cite{Bortolotti-Silva}]\label{upper.bound} It is valid that:
		\begin{enumerate}[a)]
			\item 	$\dim_H (\rho(\Delta_{T}(x)))\leq \overline{\dim}_{B}(\rho(\Delta_{T}(x)))\leq d_{0}$ for every $x \in \mathbb{T}^l$; 
			\item 	$\dim_{H}(\Delta_{T}(x))\leq \overline{\dim}_{B}(\Delta_{T}(x))\leq d_{0}$ for every $x \in \mathbb{T}^l$;
			\item	$\dim_{H}(\Delta_{T})\leq \overline{\dim}_{B}(\Delta_{T})\leq  l +d_{0}$.
		\end{enumerate}
	\end{proposition}
	
	The difficult part of  Theorem A are the lower bounds for the Hausdorff dimensions (that are enough since $\dim_H (X) \leq \dim_B (X)$ for every set $X$).
	
	\subsection{The smallest singular value}
	
	The smallest singular value shall be a
	useful tool to analyze the structure of the attractor close to the intersection between its components.
	
	\begin{definition}\label{def vs}
		Given a linear transformation $A:\mathbb{R}^{m}\to\mathbb{R}^{n}$, with $m\geq n$, the \textbf{smallest singular value} of $A$ is   
		\begin{equation}
			\textit{\textit{\textbf{m}}}(A):= \sup_{\dim(W)=n} \inf_{v\in W, \lVert v \rVert=1}\lVert A(v)\rVert,
		\end{equation} where the supremum is taken over the $n$-dimensional subspaces $W\subset \mathbb{R}^{m}$.
	\end{definition} 
	
	{Usually the singular values are defined as the non-negative square roots of the eigenvalues of $A^*A$.
		Denoting $\sigma_1 \geq \cdots \geq \sigma_m$ the singular values of $A$, they satisfy the following max-min relation:
		\begin{equation*}
			\sigma_j:= \sup_{\dim(W)=j} \inf_{v\in W, \lVert v \rVert=1}\lVert A(v)\rVert,
		\end{equation*}
		We call $\textbf{\textit{m}}(A)$ the smallest singular value because it is the smallest singular value that is not trivially null,  since $\sigma_m=\cdots = \sigma_{m-n-1}=0$ and $\sigma_{m-n}=\textbf{\textit{m}}(A)$.}
	
	{	A linear transformation $A:\R^m \to \R^n$ is surjective if and only if $\textbf{\textit{m}}(A)>0$, this means that we can understand the smallest singular value as a way to measure how surjetive the transformation is.}
	
	{For linear transformation $T_i: \R^n \to \R^{m-n}$, $i=1,2$, this means that the subspaces $E_i = \operatorname{graph}T_i$ are transversal if and only if $\textit{\textbf{m}}(T_1-T_2)>0$.
		So the value  $\textbf{\textit{m}}(T_1-T_2)$ measures how transversal are the graphs of $T_1$ and $T_2$.}
	
	One property of the smallest singular value that will be useful is the following triangular inequality:
	\begin{equation}\label{triangular}
		\textit{\textbf{m}}(A) \leq \textit{\textbf{m}}(B) + \lVert A-B\rVert.
	\end{equation}

	\section{The Geometric Lemma}
	The main step in the proof of the Theorem A is the following Geometric Lemma, which  is a higher-dimensional extension of Lemma 3.1 in \cite{Bothe}.
	
	\begin{lemma}\label{LG}
		(Geometric Lemma) There is $0<\mu_{0}<\frac{\log \overline \lambda}{2 \log \underline{ \lambda}}$ such that for any $\mu \in\big( \mu_{0} ,\frac{\log \overline \lambda}{2 \log \underline{ \lambda}} \big) $ and for each integer $n$ large, we can find an integer $k$, with $1<k<\frac{n}{2}$, and a proper compact subset $\mathbf{F}$ in $\mathbb{T} ^{l}$ which is the union of at least $sN^{n-1}-N^{\mu n}$ sets in $\mathcal{T}_{n}$ such that for any two points $x_1 , x_2 \in \mathbf{F}$ that satisfy $\varphi^{k-1}(x_1)\neq \varphi^{k-1}(x_2), \ \varphi^{k}(x_1)=\varphi^{k}(x_2)$, we have $\rho(T^{k}(\mathbb{D}(x_1)))\cap \rho(T^{k}(\mathbb{D}(x_2)))=\emptyset$.
	\end{lemma}
	
	This Section  is dedicated to the proof of this Lemma.

	\subsection{Overlaps between components of the attractor}

	Intrinsic transversality gives a uniform lower bound for 
	$\textbf{\textit{m}}( \rho DS(x,\underline{a})-\rho DS(x,\underline{b}))$
	when $\rho(S(x,\underline{a}))$ is close to $\rho(S(x,\underline{b}))$.

	\begin{proposition}\label{Lema 5.2}
		There are constants $c_{1}>0$, $\delta_1>0$ and  $k_1 \geq 1$ large enough for which the following holds for every $k\geq k_1$ (and $k\leq +\infty$):  let $\underline{a},\underline{b}\in I^{k}(x)$ be   such that $a_1\neq b_1$ and $\lVert \rho(S(x,\underline{a})) - \rho(S(x,\underline{b}))\rVert \leq \delta_1$, then it  is valid that
		\begin{equation}
			\textbf{\textit{m}}\big( \rho DS(x,\underline{a})-\rho DS(x,\underline{b})\big) > 2c_{1}.
		\end{equation}	
	\end{proposition}
	
	\begin{proof}
		Applying the Property (P1) to the graphs of $DS(x,\underline{a})$ and  $DS(x,\underline{b})$,
		intrinsic transversality implies that $\textit{\textbf{m}}(\rho DS(x,\underline{a}_\infty)- \rho DS(x, \underline{b}_\infty) )>0$ for every $(x,\underline{a}_\infty,\underline{b}_\infty) \in \tilde \Sigma$, where 
		$$  \tilde \Sigma = \{(x,\underline{a}_\infty,\underline{b}_\infty) : \rho S(x,\underline{a}_\infty) = \rho S(x,\underline{b}_\infty), \underline{a}_\infty, \underline{b}_\infty \in I^\infty(x) \text{ and } [\underline{a}_\infty]_1 \neq [\underline{b}_\infty]_1  \}.$$
		
		Since $\tilde \Sigma$ is compact, there exists a constant $c_1>0$ such that 
		\begin{equation}\label{ineq}
			\textit{\textbf{m}}(\rho DS(x,\underline{a}_\infty)- \rho DS(x, \underline{b}_\infty) )>4c_1
		\end{equation}
		for every $(x,\underline{a}_\infty,\underline{b}_\infty) \in \tilde \Sigma$. 
		
		Considering $\tilde \Sigma_m = \{(x,\underline{a}_\infty,\underline{b}_\infty) : \lVert  \rho S(x,\underline{a}_\infty) - \rho S(x,\underline{b}_\infty) \rVert < 1/m $, $\underline{a}_\infty, \underline{b}_\infty \in I^\infty(x)$ and $[\underline{a}_\infty]_1 \neq [\underline{b}_\infty]_1  \}$, it follows that there exists $m_1$ large so that \eqref{ineq} is valid for every   $(x,\underline{a}_\infty,\underline{b}_\infty) \in \tilde \Sigma_m$, $m \geq m_1$.
		In fact, if for every $m$ there exists $(x_m, \underline{a}_m, \underline{b}_m) \in \tilde \Sigma_m$ for which \eqref{ineq} is not valid, considering $(x_0, \underline{a}_0, \underline{b}_0)$ any accumullation point, then by continuity  we would have $(x_0, \underline{a}_0, \underline{b}_0) \in \tilde \Sigma$ and not satisfying \eqref{ineq}.
		So it follows the Proposition  for $k=\infty$ and $\delta= 1/m_1$.
		
		Take $k_1$ large such that 
		$||\rho S(x,c_{\infty})-\rho S(x,\underline{c})|| \leq \delta/3$ 
		and 
		$||\rho DS(x,c_{\infty})-\rho DS(x,\underline{c})||  \leq c_1$ 
		for every $c_\infty \in I^\infty(x)$, $\underline{c}=[c_\infty]_k$  and $k \geq k_1$. 
		Proposition \ref{Lema 5.2}  follows by triangular inequality and $\delta_1=\delta/3$.
	\end{proof}

	In what follows, we consider $q$ large fixed so that $\diam \mathcal{T}_{q} < \dfrac{c_1}{2\kappa}$. 

	\begin{definition}
		Given $m > 1$ and $\underline{c}\in I_{q}$, $q\geq 1$,  consider two components of $T^{m}(V)\cap \pi^{-1}(\mathbf{T}_{\underline{c}})$ that are in distinct components of $T(V)\cap \pi^{-1}(\mathbf{T}_{\underline{c}})$,
		that is, there exist words $\underline{a}, \underline{b}\in I_m$ with $a_{1}\neq b_{1}$ such that 
		$Z_{\underline{a}}= T^m(\mathbf{T}_{\underline{a}} \times E \times F) $ and $Z_{\underline{b}}= T^m(\mathbf{T}_{\underline{b}} \times E \times F)$. 
		We say that the pair $(Z_{\underline{a}}, Z_{\underline{b}})$ is an \textbf{overlap} of $T^{m}(V)$ over $\mathbf{T}_{\underline{c}}$ if $\rho(Z_{\underline{a}})\cap \rho(Z_{\underline{b}}) \cap  \pi^{-1}(\mathbf{T}_{\underline{c}}) \neq \emptyset$. 
		%
		The set $B= \pi (\rho(Z_ {\underline{a}})\cap \rho(Z_{\underline{b}})) \subset \mathbb{T}^{l}$  will be called the \textbf{$\pi-$projection of the overlap  of the pair $(Z_{\underline{a}}, Z_{\underline{b}})$}.
	\end{definition}

	\begin{proposition}\label{Prop 5.1} For every $k\geq k_{1}$ it is valid the following:
		if $D_1$ and $D_2 $ are components of $T^{k}(V)\cap\mathbb{D}(x_{0})$ that are in different components of $T(V)\cap \mathbb{ D}(x_{0})$ and for which $\rho(D_1)\cap \rho(D_2)\neq \emptyset$, then there is a overlap $(Z_1,Z_2)$  in $ T^{k}(V)$ over some $\mathbf{T}_{\underline{c}}$, with $\underline{c}\in I_{q}$, $D_1\subset Z_1$ and $D_2\subset Z_2$. Moreover, we have 
		\begin{equation}
			\textbf{\textit{m}}\big( \rho DS(x,\underline{a})-\rho DS(x,\underline{b})\big) > c_1 \ \ \ \forall \ \ x\in\mathbf{T}_{\underline{c}}.
		\end{equation}
	\end{proposition} 
	\begin{proof}
		Consider $\underline{a}, \underline{b}\in I^{k} (x_{0})$, with $[\underline{a}]_1\neq [\underline{b}]_1$, that describe the components $D_1$ and $D_2$, respectively. Consider the corresponding pair $(Z_{\underline{a}}, Z_{\underline{b}})$, and
		$t_0=\rho(S(x_{0},\underline{a}))=\rho(S(x_{0},\underline{b}))  \in\rho(\Delta_{T}(x_{0}))$,
		in particular $D_1\subset Z_{\underline{a}}$, $D_2\subset Z_{\underline{b}}$
		and Proposition \ref{Lema 5.2} implies
		$\textbf{\textit{m}}\big( \rho DS(x_{0},\underline{a})-\rho DS(x_{0},\underline{b})\big) \geq 2 c_1$.
		
		Let $\underline{c}\in I_{q}$ be so that $x_{0}\in \mathbf{T}_{\underline{c}} $, by triangular inequality we have:	
		\begin{align*}
			\textbf{\textit{m}}(\rho DS(x,\underline{a})-\rho DS(x,\underline{b}))\geq \textbf{\textit{m}}(\rho DS(x_0,\underline{a})-\rho DS(x,\underline{b}_0))-2\kappa||x - x_{0}||> c_1
		\end{align*}
		for all   $x\in\mathbf{T}_{\underline{c}}$. Then  $(Z_{\underline{a}},Z_{\underline{b}})$ is an overlap of $T^{k}(V)$ over $\textbf{T}_{\underline{c}}$ as we want.

	\end{proof}

	
	
	
	Let us analyze the $\pi$-projection of the overlap and estimate how small it is. 

	\begin{proposition}\label{Prop 5.2}
		Given $\omega \in (\overline{\lambda},1)$, there is a constant $K_1>0$  and an integer $k_2$  such that for every $k\geq k_{2}$ and every overlap in $T^{k}(V)$,
		there is a submanifold $S  \subset \mathbb{T}^{l}  $ of dimension $l-p$ and diameter at most $K_1$ such that 
		the $\pi$-projection of this overlap    is contained in the tubular neighborhood of $S$ of radius of bounded from above by $\omega^ {k }$.
	\end{proposition}
	\begin{proof}
		Let $(Z_{\underline{a}},Z_{\underline{b}})$ be an overlap of $T^{k}(V)$ over some $\mathbf{T}_{\underline{c}},  \underline{c}\in I^{q}$.
		Fix $x_{0}\in \pi(\rho(Z_{\underline{a}})\cap \rho(Z_{\underline{b}})) $,  $r_{0}<\min\{\frac{c_{0}}{ 4 \kappa}, \gamma \}$
		and consider $
		g:\mathbf{B}(x_{0},r_{0})\to \mathbb{R}^{p}$ given by
		$$g(x)=\rho S(x,\underline{a})-\rho S(x,\underline{b}),$$	
		where $g$ is defined considering the extensions of $S(x,\underline{a})$ and $S(x,\underline{b})$ to $\mathbf{B}(x_{0}, \gamma)$.  In particular, $g(x_0)=0$. 
		Proposition \ref{Prop 5.1} implies that $g$ is a submersion,
		so $S=  g^{-1}(0)$ is a submanifold of dimension $l-p$ with diameter at most $K_1>0$.
		%
		
		%

		Notice that $\rho(Z_{\underline{c}})\subset\{(x,y)\in\mathbb{T}^{l}\times [-1,1]^{p}:||y-\rho S(x,\underline{c})||\leq \overline{\lambda}^{k} \}$, so 
		%
		if $x\in \pi(\rho(Z_{\underline{a}})\cap \rho(Z_{\underline{b}})) $, then  $||\rho S(x,\underline{a})-\rho S(x,\underline{b})||\leq 2\overline{\lambda}^{k}$, that is, 
		the $\pi$-projection of the overlap $\pi(\rho(Z_{\underline{a}})\cap \rho(Z_{\underline{b}}))$ is contained in $g^{-1}(\mathbf{B }(0,2\overline{\lambda}^{k}))$. 
		
		From the local form of the submersions, we can notice that $g^{-1}(\mathbf{B }(0,2\overline{\lambda}^{k}))$ is contained in a tubular neighborhood of $S$ of radius bounded from above by $K_2  \overline{\lambda}^k$. $K_1$ and $K_2$ can be taken uniform (independent of $k, \underline{a}, \underline{b}, S, x_0$) since $||Dg||$ is uniformly bounded from above and $\textit{\textbf{m}}(Dg)$ is uniformly bounded from below for all $(\underline{a},\underline {b})$ corresponding to overlaps of $T^{k}(V)$ with $k \geq k_1$. The result follows since $K_2 \overline{\lambda}^k < \omega^k$ for every $k \geq k_2$  large enough.

		\end{proof}

		\subsection{Proof of the Geometric Lemma}
		\begin{proof}[Proof of the Geometric Lemma]
			$T \in \mathcal{T}^*$ implies that $N^{2}\underline{\beta}^{-l}\overline{\lambda}^{p}<1$ and

			\begin{equation}\label{eq.mu}
				\mu_{0}:= \dfrac{\dfrac{l\log\overline{\beta}-  p\log\underline{\beta}}{2\log N}- \dfrac{l\log\overline{\beta}}{\log(\underline{\beta}^{-l }\overline{\lambda}^{p}N^{2})} }{\dfrac{1}{2}- \dfrac{2\log  N}{2\log(\underline{\beta}^{- l}\overline{\lambda}^{p}N^{2})} }	< \dfrac{\log\overline{\lambda}}{2\log\underline{\lambda}} < \dfrac{1}{2}.
			\end{equation}

			So $\mu_{0}>\dfrac{l\log\overline{\beta}- p \log\underline{\beta}}{2\log N}>0$ since $T \in \mathcal{E}^*$, and 
			for any $\mu > \mu_{0} $ it is valid that   


			\begin{equation} \label{5.3 eq}
				\dfrac{\mu\log N-l\log\overline{\beta}}{\log(N^{2}\underline{\beta}^{-l}\overline{\lambda}^{p} )} + \dfrac{l\log\overline{\beta}- p \log\underline{\beta}}{2\log N}< \dfrac{\mu}{2}. 
			\end{equation} 
			
			Consider $\omega\in (\overline{\lambda},1)$  close to $\overline{\lambda}$
			such that
			$N^2 \underline{\beta}^{-p}\omega^p<1$ and 
			\begin{equation} \label{5.3 eq.}
				\dfrac{\mu\log N-l\log\overline{\beta}}{\log(N^{2}\underline{\beta}^{-l}\omega^{p} )} + \dfrac{l\log\overline{\beta}- p \log\underline{\beta}}{2\log N}< \dfrac{\mu}{2}. 
			\end{equation} 
			
			Fix an integer $k_{3} = \max\{ k_1, k_2\}$. Given $k\geq k_{3}$, consider $B_1,B_2,...,B_{s^*}$ as the $\pi$-projections of the overlaps of $T^{k}(V)$.
			The number $s^*$  of overlaps of $T^{k}(V)$ is at most $s^2 N^{2k} $. 
			
			
			For each $B_{i}$, $i=1,...,s^*$, take $B^{*}_{i}$ such that $B^{*}_{i}= \varphi^{-k}(B_{i})$, where $\varphi^{-k}$ is one of the inverse branches of $\varphi^{k}$ corresponding to the overlap of $T^{ k}(V)$ associated to $B_i$.
			Each $B_{i}$ is contained in a tubular neighborhood $V_{i}$ of some submanifold $S_{i}$ of dimension $l-p$ and radius bounded from above by $\omega^{k }$.

			%
			%
			
			%
			
			Consider $\hat V_i $ the neighborhood of $V_i$ of radius $\underline{\beta}^{-(n-k)}$ and $V_i^{**} = \varphi^{-k}(\hat V_i^{*})$.
			
			Let $\mathcal{R}_{1i},\mathcal{R}_{2i},...,\mathcal{R}_{mi} $ be the   rectangles of $\mathcal{T}_{n}$ such that $\mathcal{R}_{ji}\cap B^{*}_{i}\neq\emptyset,$  
			then: 
			\begin{equation*}
				\mathcal{R}_{ji} \cap B_i^{*} \neq \emptyset \Rightarrow 
				\varphi^k(\mathcal{R}_{ji}) \cap B_i \neq \emptyset \Rightarrow 
				\varphi^k(\mathcal{R}_{ji}) \subset \hat V_i \Rightarrow 
				\mathcal{R}_{ji} \subset \varphi^{-k}(\hat V_i) = V_i^{**}.
			\end{equation*}
			This implies  that  $\cup_{j=1}^{m} \mathcal{R}_{ji} \subset V^{**}_{i}$    and  $\sum_{j=1}^{m}\operatorname{vol}(\mathcal{R}_{ji})\leq\operatorname{vol} (V^{**}_{i})$. 
			
			Notice that $\hat V_i$ is contained in a tubular neighborhoud of a submanifold with diameter uniformly bounded  and with radius at most $\omega^k + \underline{\beta}^{-(n-k)}$, so we have the following estimates:
			\begin{equation}  \operatorname{vol}(\hat V_i ) \leq K_3 (\omega^k + \underline{\beta}^{-(n-k)})^p	\end{equation}
			for some constant $K_2$.
			We will also use that $ \operatorname{vol}(V_i^{**}) \leq \underline{\beta}^{-kl}  \operatorname{vol}(\hat V_i )$.

			For $\mathcal{R}\in\mathcal{T}_{ n}$, we have $ K_{0}\overline{\beta}^{-n} \leq \operatorname{diam} \mathcal{R}\leq \underline{\beta}^{-n }$.
			So
			%
			$ mK_{0}\overline{\beta}^{-ln}  \leq \sum_{j=1}^{m}\operatorname{vol}(\mathcal{R}_{ji}),$ what implies:
			\begin{equation} m \leq  K^{-1}_{0} \overline{\beta}^{kn} \operatorname{vol}(V_i^{**})  
					\leq  K_3\overline{\beta}^{ln}  \underline{\beta}^{-kl} (  \omega^{kp}+ \underline{\beta}^{-(n-k)p} )			\end{equation}
				for some $K_3>0$.
				
				Considering $B^{*}= \bigcup_{i=1}^{s^{*}}B^{*}_{i}$,
				the number $M$ of rectangles in $\mathcal{T}_{n}$ that intersect $B^{*}$ is bounded from above by 
				\begin{align} 
					M\leq m  s^{*} 
					\leq K_4   ( \overline{\beta}^{ln}\underline{\beta}^{-kl}\omega^{kp}N^{2k} + \overline{\beta}^{ln}\underline{\beta}^{- kl -(n-k)p } N^{2k})
				\end{align}
				for  constant $K_4>0$.   
				%
				Therefore, the number of rectangles in $\mathcal{T}_{n}$ that do not intersect $B^{*}$ is at least 	\begin{equation}sN^{n-1}- K_4 ( \overline{\beta}^{ln}\underline{\beta}^{-kl}\omega^{kp}N^{2k} +  \overline{\beta}^{ln}\underline{\beta}^{- kl -(n-k)p } N^{2k}). 	\end{equation}
				
				Let $\mathbf{F}$ be the closure of the union of rectangles in $\mathcal{T}_n$ that do not intersect $B^{*}$. We have that $\mathbf{F}$ is compact and that for any $x_1, x_2\in F$ it holds that if $\varphi^{k-1}(x_1)\neq \varphi^{k-1}( x_2)$ and $ \varphi^{k}(x_1)=\varphi^{k}(x_2)$, then $\rho(T^{k}(D(x_1)))\cap \rho (T^{k}(D(x_2)))=\emptyset,$ otherwise we would have $\varphi^{k}(x_1)$ in some $B_{i}$.
				
				Now to conclude the Lemma, we need to show that for sufficiently large $n$ we can have $k<\frac{n}{2}$  for which it holds \begin{equation}\label{eq.5.4}
					K_4 \overline{\beta}^{ln}\underline{\beta}^{-k l }\omega^{pk}N^{2k}+ K_4  \overline{\beta}^{ln}\underline{\beta}^{- kl -(n-k)p } N^{2k}. 
				\end{equation}  
				
				Since $N^2\underline{\beta}^{- l}\omega^{p}<1,$ we have: 
				\begin{equation*} 
					K_4\overline{\beta}^{ln}\underline{\beta}^{-k  l}\omega^{pk}N^{2k}< \dfrac{N^{n\mu}}{2}
					\Longleftrightarrow  k>n\dfrac{(\mu\log N-l\log\overline{\beta})}{\log(N^2\underline{\beta}^{-  l}\omega^{p})}-  \dfrac{\log(2K_4)  }{\log(N^2\underline{\beta}^{- l}\omega^{p})}  
				\end{equation*}
				and \begin{equation*} 
					K_4  \overline{\beta}^{ln}\underline{\beta}^{- kl -(n-k)p } N^{2k} < \dfrac{N^{n\mu}}{2}
					\Longleftrightarrow k< n\Big(  \dfrac{\mu}{2}-\dfrac{(l\log\overline{\beta}-p\log\underline{\beta})}{2\log N - (l-p)\log \underline\beta } \Big) - K_5    
				\end{equation*}
				for $K_5 = \frac{\log(2K_4)}{2\log N - (l-p)\log \underline\beta  }$.
				
				Therefore, \eqref{eq.5.4} is valid if we  consider $n$ and $k$ with  \begin{equation*}\label{eq.5.5}
					n\dfrac{(\mu\log N-l\log\overline{\beta})}{\log(N^2\underline{\beta}^{- l}\omega^{p})}-\dfrac{\log(2K_4)  }{\log(N^2\underline{\beta}^{- l}\omega^{p})}< k <  n\Big(  \dfrac{\mu}{2}-\dfrac{(l\log\overline{\beta}-p\log\underline{\beta})}{2\log N} \Big)-   K_5.  
				\end{equation*}
				
				
				By taking $n $ large enough, the difference between the right-hand expression and the left-hand expression above is greater than 1. So there exist $k$ satisfying this inequality and $k<n/2$. 

			\end{proof}


			\section{Consequences of the Geometric Lemma}
			The rest of the proof of the Theorem \ref{Teo 5.1} is similar to the proof of \cite[Theorem A]{Bothe}, we describe it in this Section for completeness. 
			
			The Geometric Lemma allows us to construct sub-attractors that we know how to estimate a lower bound  of its dimension, and the restriction of the unstable holonomy to the $\rho$-projections of these sub-attractors is locally  bi-Lipschitz, this will allow to extend the lower bound for every sable section.
			
			
			\subsection{Subsets with large dimension}

			Consider $m$ a large integer so that Lemma \ref{LG} holds for a fixed value $\mu\in (\mu_{0},\frac{\log\overline\lambda }{2 \log \underline\lambda }) $ and $n=2m$, it gives an integer $k=k(m)<m$ and a compact set $\mathbf{F}=\mathbf{F}(m)\subset \mathbb{T}^{l}$.

			Let $\underline{a}_1, \underline{a}_2, ...,\underline{a}_{t}$ be the words in $I_n$ for which $\mathbf{T}_ {\underline{a}_1}, \mathbf{T}_{\underline{a}_2},..., \mathbf{T}_{\underline{a}_t} \in \mathcal{T} _n$ are not in $\mathbf{F}$. Since $\mathbf{F}$ is the union of at least $sN^{n-1}-N^{\mu n}$ sets in $\mathcal{T}_n$, it follows that $1\leq t\leq N^{\mu n}=N^{2\mu m}$.
			
			\begin{definition}
				Given a word $\underline{a}=(a_1,a_2,..., a_{m})\in I_m$, we say that \textbf{$\underline{a}$ appears in $\underline{\hat{a}}\in I_n$ } if $n\geq m$ and exists $j_{0}\in\{0,1,...,n-m\}$ such that $\hat{a}_{j_{0}+j}=a_{j}$ for $j=1,...,m$.
			\end{definition}
			
			Define 
			$I'_{m}=\{\underline{a}\in I_m: \underline{a} \mbox{ does not appear in }\ \underline{a }_{j}, \ j=1,2,...,t. \}.
			$

			For every $\underline{a}_{j}, \ j=1,2,...,t$, we cannot have all $\pi_{m}(\underline{a}_{j}), \pi_{m}(\sigma(\underline{a}_{j})), ...,  \pi_{m}(\sigma^{m}(\underline{a}_{j}) )$ in $I'_{m}$, so the number $r=r(m) = \# I'_{m}$ satisfies $sN^ {m-1}-(m+1)t\leq r < sN^{m-1}$.
			Changing $\mu$ by a slightly smaller number and $m$ large, we can suppose that 
			%
			\begin{equation}
				sN^{m-1}-N^{2\mu m}< r< sN^{m-1}.
			\end{equation}

			For each $u=1,2,...$, consider $I_{u m}$  as subset of $I_m\times I_{m }\times ...\times I_m$ 
			and 	$I'_{u,m} = (I'_{m}\times ...\times I'_m) \cap I_{um}.$

			
			

			Consider  $I'_{\infty, m}:=\{\underline{a}\in I_{\infty}: \pi_{um}(\underline{a})\in I'_{u,m}, \  \ u=1,2,...\}$ and 
			notice that if $\underline{a}\in I'_{\infty, m}$ then none of the words $\underline{a}_1, \underline{a}_2, ... ,\underline{a}_{t}$ appear in $\underline{a}$.
			%
			
			Consider also the sets:
			\begin{itemize}
				\item $C'=C'(m)=\bigcap_{u=1}^{\infty} \big(\bigcup_{\underline{a }\in I'_{u,m}}\mathbf{T}_{\underline{a}} \big) \subset \mathbb{T}^{l}$;
				
				\item $V'=V'(m)=C'(m)\times E\times F=\pi^{-1}(C'(m))$;
				
				\item $\Delta'_{T}=\Delta_{T}'(m)=\bigcap_{u=1}^{\infty}T^{um}(V')\subset \Delta_{T}\cap V'$;
				
				\item $\Delta'_{T}(x)=\Delta'_{T}(m,x)=\Delta_{T}'(m)\cap \mathbb{D}(x)$. 
			\end{itemize}
			
			It is valid that $\varphi^{m}(C')=C'$ and,  since every $\underline{a}_{j}$  does not appear in any $\underline{a}\in I'_{\infty, m }$, it follows that $\varphi^{i}(C')\subset \mathbf{F}$ for all $i\geq 0$. 
			Notice also that $\tau(I'_{\infty,m}) = C^{'}(m) $ and $\sigma^m(I'_{\infty,m}) = I'_{\infty,m}$.

			\begin{theorem}\label{TeoB1}
				It is valid that
				$$	\liminf_{m\to\infty}\inf_{x'\in C'(m)}(\dim_{H}\rho(\Delta_{T} '(x')))\geq d_{0}.$$
				Moreover, the restriction $\rho|_{\Delta_{T } '(x)}$ is injective for every $x \in C'(m)$.
			\end{theorem}
			%
			Given $x'=\tau(\underline{a}')\in C'(m)$, with $\underline{a}'\in I'_{\infty, m}$,
			denote 
			\begin{equation}I^{'}_{u,m}(x')=\{\underline{a}\in I'_{u,m}: \underline{a}\underline{a}'\in I'_{\infty, m}\}.	\end{equation}
			{We denote $I_m(x) = I_{1,m}(x)$ and $I'_m(x) = I'_{1,m}(x)$.}
			
			For $\underline{a}\in I_{um}(x)$, denote also
			\begin{equation}
				D_{x,\underline{a}}:= \rho {T}^{um}_{[\underline{a}\underline{a}']_{um}(x)}(\rho(\mathbb{D}([\underline{a}\underline{a}']_{um}(x))))
			\end{equation}

			Define the positive real numbers $\textbf{d}(x',m)$ and   $\textbf{d}^{*}(x',m)$  by the equalities
			\begin{align}
				\sum_{\underline{a}\in I_{m}(x')}\operatorname{diam}(D_{x',\underline{a}})^{ \textbf{d}(x',m)}&=1 \label{d}\\
				\sum_{\underline{a}\in I^{'}_{m}(x')}\operatorname{diam}(D_{x',\underline{a}})^{ \textbf{d}^{*}(x',m)}&=1 \label{d*}
			\end{align}

			For every $x' \in C'(m)$, it is valid that: 
			\begin{equation}\label{ineq}
				\textbf{d}(x',   m)\geq \textbf{d}^{*}(x',  m)
			\end{equation}
			
			We also have for every $x \in \mathbb{T}^l$:
			\begin{equation}\label{d_0}
				d_0 = \underset{ m \to + \infty}{\lim} \mathbf{d}(x,   m).					
			\end{equation}
				
				{					Notice that $\# I_{m}(x') = N^{m}$ and } 
				\begin{equation}\label{contagem}
					{					\# I^{'}_{m}(x') \geq N^{m} - N^{2\mu m} }
				\end{equation}

			\begin{proposition}\label{Lema 5.4} It  is valid that 
				$  \underset{m\to\infty}{\lim} \underset{x \in C'(m)}{\inf}  \textbf{d}^{*}(x,m) = d_0$. 
			\end{proposition}
			
			\begin{proof}
				We prove  by contradiction. By \eqref{ineq} and \eqref{d_0}, if the Proposition is not valid then we can suppose that 
				$d_0-\textbf{d}^{*}(x_m,m)>\alpha>0$ for infinitely many integers $m$ and some point $x_m \in C'(m)$. For these  $m$ and $x_m$ we have:
				\begin{align*}
					1 &=  \sum_{\underline{a}\in  I'_m(x_m)}\operatorname{diam}(D_{x_m,\underline{a}})^{ \textbf{d}(x_m,m)} + \sum_{\underline{a}\in  I_m(x_m) - I'_m(x_m) }\operatorname{diam}(D_{x_m,\underline{a}})^{ \textbf{d}(x_m,m)}\\
					&\leq 
					\overline{\lambda}^{m  \alpha } \sum_{\underline{a}\in  I'_m(x_m)}\operatorname{diam}(D_{x_m,\underline{a}})^{ \textbf{d}^*(x_m,m)}     + N^{2\mu m}   \overline{\lambda}^{m \textbf{d}(x_m,m)} \\
					&= \overline{\lambda}^{m  \alpha } +  N^{2\mu m}   \overline{\lambda}^{m \textbf{d}(x_m,m)}
				\end{align*}
				%
				
				Since $\mu < \dfrac{\log \overline\lambda}{2\log \underline \lambda}$, we have that  $$d_0=  \limsup_{m\to +\infty} \textbf{d}(x_m,m)\leq \frac{2\mu\log N}{-\log \overline\lambda } < \frac{\log N}{-\log \underline\lambda }.$$

				On the other hand, we also have that
				\begin{align*}
					1& =  \sum_{\underline{a}\in  I'_m(x_m)}\operatorname{diam}(D_{x_m,\underline{a}})^{ \textbf{d}(x_m,m)} + \sum_{\underline{a}\in  I_m(x_m) - I'_m(x_m) }\operatorname{diam}(D_{x_m,\underline{a}})^{ \textbf{d}(x_m,m)}  \\
					&\geq \big( N^{m}- (N^{m}-N^{2\mu m}) \big) \underline\lambda ^{m\textbf{d}^{*}(x_m,m)}   \\
					&= N^{2\mu m} \underline\lambda^{m\textbf{d}^{*}(x_m,m)}
				\end{align*}

				This   implies that 
				$$
				\liminf_{m\to\infty}\textbf{d}^{*}(x_m,m)\geq \frac{\log N}{-\log\underline\lambda}
				$$
				
				Therefore 
				$    \dfrac{\log N}{-\log \underline\lambda }\geq  \underset{m\to\infty}{\liminf}\textbf{d}^{*}(x_m,m)> d_0 + \alpha >  \dfrac{\log N }{-\log\underline\lambda},$
				what is a contradiction.
				
				%
				
			\end{proof}

			In the following, we need some estimates of bounded distortion for $D_y\nu$.  Denote 
			$T^n(x,y,z) 
			=  (\varphi^n(x),  \nu^{n}(x,y),\psi^{n}(x,y,z))$ 
			and  $\tilde{T}^{n}_{x}(y)= \nu^{n}(x,y)$.  
			
			
			\begin{lemma}\label{lema 3.1}
				There exists a constant  $K_6>0$ such that for every $n\geq 1$, $x \in \mathbb{T}^l$ and $\underline{a}\in I^{\infty}(x)$, it is valid that: 
				if $(\varphi^{n-k}(x),y_{k})$ and $(\varphi^{n-k}(x),y_k^{*})$ are in the convex hull of $\tilde{T}^{n-k}_{[\underline{a}]_{n}(x)}(\mathbb{D}([\underline{a}]_{n}(x)))$ for every $k=0,1,...,n-1$, then:
				\begin{equation}
					K_6^{-1}  \leq 	\prod_{k=0}^{n-1}\dfrac{\lambda(\varphi^{n-k}(x),y_{k})}{\lambda(\varphi^{n-k}(x),y^{*}_{k})}\leq K_6
				\end{equation}
				
				{Considering $(x_n,y_n) \in \mathbb{D}([\underline{a}]_n(x) )$ and  $(x_k,y_k) = \tilde T^{n-k}_{x_n}(x_n,y_n) $, it also follows that}
				\begin{equation}\label{diam}
					{	K_6^{-1}  \leq 	\dfrac{\prod_{k=0}^{n-1}\lambda(\varphi^{n-k}(x),y_{k})}{ \diam D_{x,\underline{a} }}\leq K_6 }
				\end{equation}
				
		\end{lemma}
		\begin{proof}
			By conformality, we have that $$\lambda(x,y)=\lVert D_{y}\nu(x,y) |_{Y} \rVert { = \lVert D_{y}\nu(x,y)^{-1} |_{Y} \rVert ^{-1} } = |\det D_y\nu (x,y)_{|_{Y}}|^{1/p},$$ so the  standard estimates of distortion for $|\det D_y\nu|$ are also valid for $\lambda$
		\end{proof}
		
		{	Let us also consider a positive real numbers $\overline{\lambda}_{x,\underline{a}}$ and  $ \textbf{d}^{**}(x',m)$ defined by}
		\begin{equation}
			{\overline{\lambda}_{x',\underline{a}}:= \sup_{p \in \rho(\mathbb{D}( \underline{a}  (x') ))  } \| D_yT^{m}_{ \underline{a}   (x')} (p) \| \quad , \quad  \text{   for  } \underline{a} \in I_m(x')}  
		\end{equation} 
		and 	
		\begin{equation}\label{d**}
			{\sum_{\underline{a}\in I^{'}_{m}(x')} \overline{\lambda}_{x',\underline{a}}^{ \textbf{d}^{**}(x',m)}= K_6} 	
		\end{equation}
		
		\begin{proposition} { It  is valid that 
				$  \underset{m\to\infty}{\lim}  \underset{x \in C'(m)}{\inf}  \textbf{d}^{**}(x, m) = d_0 $.}
		\end{proposition}
		\begin{proof}
			{ By Lemma \ref{lema 3.1} it follows that  
				$$K_6^{-1} \leq \frac{ \overline{\lambda}_{x,\underline{a}} }{ \diam D_{x,\underline{a}}} \leq K_6.$$ }
			
			{Since $\underset{m\to \infty}{\lim} \underset{\underline{a} \in I'_m(x')}{\inf} \diam D_{x,\underline{a} }  =0 $, it follows from  \eqref{d*} and \eqref{d**} that  $$\lim_{m\to \infty}\underset{x \in C'(m)}{\inf} \textbf{d}^{**}(x,m) = \lim_{m\to \infty}\underset{x \in C'(m)}{\inf} \textbf{d}^*(x,m) = d_0.$$
			}	
		\end{proof}

		\begin{proposition}\label{Prop 5.6} For every $m$, there exist  $\delta_1=\delta_1(m)$ so that:  
			for every $x_1\neq x_2$ in $C'(m)$ with  $\varphi^{m}(x_1)=\varphi^{m}(x_2)$,
			the distance between 
			$\rho(T^{m}(\Delta'_{T}(x_1)))$ and $\rho(T^{m}(\Delta'_{T}(x_2)))$ is at least $\delta_1$. Moreover, if $x\in C'$ then $\rho|_{\Delta'_{T}(x)}$ is injective.
		\end{proposition}

		\begin{proof} 
			Consider
			$k$ and $\mathbf{F}$ as in Lemma \ref{LG}, for every $x\in\mathbb{T}^{l}$ and any two components $D_1$, $D_2$ of $T(V)\cap \mathbb{D}(x)$ we have 	\begin{equation*}\rho(D_{1}\cap T^{k}(\pi ^{-1}(\mathbf{F})))\cap \rho(D_{2}\cap T^{k}(\pi ^{-1}(\mathbf{F})))= \emptyset .	\end{equation*}
			
			Therefore we can take $\delta = \delta (m)>0$  such that 
			\begin{equation}\label{Prop 5.5}
				d(\rho(T^{k}(\mathbb{D}(x_1))), \rho(T^{k}(\mathbb{D}(x_2))))\geq \delta
			\end{equation}
			for every
			$x_1,x_2\in \mathbf{F}$ such that $\varphi^{k-1}(x_1)\neq \varphi^{k-1}(x_2)$ and $\varphi^{k}(x_1)=\varphi^{k}(x_2)$. We can also have that:  
			\begin{equation*}d(\rho(D_1\cap T^{k}(\pi^{-1}(\mathbf{F}))), \rho(D_2\cap T^{k}(\pi^{-1}(\mathbf{F}))) )\geq \delta 	\end{equation*}
			for any $D_1, D_2$ two distinct components of $T(V)\cap \mathbb{D}(x), \ x\in\mathbb{T}^{l}$. 
			
			Since $\varphi^{m}(C')=C'$, there exist $y_1, y_2 \in C'$ such that $\varphi^{m}(y_1)=x_1$ and $\varphi^{m}(y_2)=x_2$. Then we have  $\varphi^{2m}(y_1)=\varphi^{2m}(y_2)$, and using \eqref{Prop 5.5}, there exists $\delta_1>0$ such that 	\begin{equation*}d(\rho(T^{2m}(\mathbb{D}(y_1))),\rho(T^{2m}(\mathbb{D}(y_2))))\geq \delta_1.	\end{equation*} 
			
			Since  $\rho(T^{m}(\Delta'_{T}(\varphi^{m}(y_{1})))\subset\rho(T^{2m}(\mathbb{D}(y_{1})))$ 
			and $\rho(T^{m}(\Delta'_{T}(\varphi^{m}(y_{2})))\subset\rho(T^{2m}(\mathbb{D}(y_{2})))$, it follows that: 	\begin{equation*}d(\rho(T^{m}(\Delta'_{T}(x_{1}))),\rho(T^{m}(\Delta'_{T}(x_{2}))))\geq \delta_1.	\end{equation*}
			
			Now, given 
			$u_1\neq u_2$ and $\Delta'(x)$, with $x\in C'$.  Consider $j_{0}$ the smallest integer so that  $x_{1}=\pi(T^{-j_{0}m}(u_1))\neq\pi(T^{-j_{0}m}(u_2))=x_2$.  Then: 	\begin{equation*}\varphi^{m}(x_1)=\pi(T^{-(j_{0}-1)m}(u_{1}))=\pi (T^{-(j_{0}-1)m}(u_2))=\varphi^{m}(x_2).	\end{equation*}
			
			It follows that: 	\begin{equation*}d(\rho(T^{-(j_{0}-1)m}(u_1)), \rho(T^{-(j_{0}-1)m}(u_2)))\geq \delta_1,	\end{equation*} and in particular:
			\begin{equation*}\rho(T^{-(j_{0}-1)m}(u_1)\neq \rho(T^{-(j_{0}-1)m}(u_2)).	\end{equation*} 
			
			Define $\hat{T}^{(j_{0}-1)m}:\rho(\Delta'_{T}(\varphi^{m}(x_{1})))\to\rho (\Delta'_{T}(x)) $ given by  $\hat{T}^{(j_{0}-1)m}(\rho(u))=\rho(T^{(j_{0}-1)m}(u)),$ where $u\in\Delta'_{T}(\varphi^{m}(x_1))$.  Notice that $\hat{T}^{(j_{0}-1)m}$ is a  bijection. So:
			\begin{align*}
				\rho(T^{-(j_{0}-1)m}(u_1)&\neq \rho(T^{-(j_{0}-1)m}(u_2))\\
				\Longrightarrow\hat{T}^{(j_{0}-1)m}(\rho(T^{-(j_{0}-1)m}(u_1))&\neq \hat{T}^{(j_{0}-1)m}( \rho(T^{-(j_{0}-1)m}(u_2)))\\ \Longrightarrow \rho(T^{(j_{0}-1)m}(T^{-(j_{0}-1)m}(u_1))&\neq \rho(T^{(j_{0}-1)m}(T^{-(j_{0}-1)m}(u_2))\\
				\Longrightarrow \rho(u_1)&\neq\rho(u_{2}).
			\end{align*}
			
		\end{proof}

		\begin{proposition}\label{Prop 5.7}
			If $\underline{a}_{1}$ and $\underline{a}_{2}$ are in $I'_{u,m}(x')$, with $u\geq 2$, and  $\sigma^{(j-1)m}(\underline{a}_1)\neq \sigma^{(j-1)m}(\underline{a}_2) $, $ \sigma^{jm}(\underline{a}_1)=\sigma^{jm}(\underline{a}_2)$ for some $j\in\{2,...,u\}, $ then 	\begin{equation}d(I_{\underline{a}_{1}},I_{\underline{a}_{2}})\geq   \delta_2    \operatorname{diam}(I_{\sigma^{jm}(\underline{a}_{1})})	\end{equation} for some constant  $\delta_2=\delta_2(m)>0$. 
		\end{proposition}

		\begin{proof}
			We have   $I_{\sigma^{jm}(\underline{a}_1)}=I_{\sigma^{jm}(\underline{a}_2)}= \rho(T^{(u-j)m}(\mathbb{D}(\tau(\sigma^{jm}(\underline{a}_1),\underline{a}'))))$. 
			
			Consider $u_1,u_2\in \mathbb{D}(\tau(\sigma^{jm}(\underline{a}_1),\underline{a}'))) $ such that the diameter of $I_{\sigma^{jm}(\underline{a}_{1})}$ is equals to $d(\rho(T^{(u-j)m}(u_1),\rho(T^{(u-j)m}(u_{2})))$.  Then
			\begin{equation*}\operatorname{diam}(I_{\sigma^{jm}(\underline{a}_{1})}) \leq\Big( \prod_{k=1}^{(u-j)m-1}\lambda(t_{k})\Big)  d(u_1,u_2)	\end{equation*} for some $t_{k}\in \rho(T^{(u-j)m-k}(\mathbb{D}(\tau(\sigma^{jm}(\underline{a}_1),\underline{a}'))))$.
			Since $d(u_1,u_2)\leq 2$, it follows that: 
			\begin{equation}\label{eq.5.6}
				\prod_{k=1}^{(u-j)m-1}\lambda(t_{k})  \geq  \dfrac{1}{2}   \operatorname{diam}(I_{\sigma^{jm}(\underline{a}_{1})}).
			\end{equation}
			
			Now, consider $x_1=\pi(\mathbb{D}(\tau(\sigma^{(j-1)m}(\underline{a_{1}}),\underline{a}')))$ and $x_2=\pi(\mathbb{D}(\tau(\sigma^{(j-1)m}(\underline{a_{2}}),\underline{a}')))$.  As $\sigma^{(j-1)m}(\underline{a}_1)\neq \sigma^{(j-1)m}(\underline{a}_2) $, we have that $x_1\neq x_2$ and as $\sigma^{jm}(\underline{a}_1)=\sigma^{jm}(\underline{a}_2)$, we have $\varphi^{m}(x_1)=\varphi^{m}(x_2)$.  Then, applying \eqref{Prop 5.5}, we can conclude that 
			
			\begin{equation*}d(\rho(T^{m}(\mathbb{D}(\tau(\sigma^{(j-1)m}(\underline{a_{1}}),\underline{a}')))),\rho(T^{m} (\mathbb{D}(\tau(\sigma^{(j-1)m}(\underline{a_{2}}),\underline{a}')))))\geq \delta_1.	\end{equation*}
			This gives that 
			\begin{equation*}d(\rho(T^{jm}(\mathbb{D}(\tau(\underline{a}_{1},\underline{a}'))),\rho(T^{jm}(\mathbb{D}(\tau(\underline{a}_{2},\underline{a}'))))\geq \delta_1.	\end{equation*}
			
			On the other hand, we have that
			\begin{equation*}d(I_{\underline{a}_{1}},I_{\underline{a}_{2}})\geq\Big( \prod_{k=1}^{(u-j)m-1}\lambda(w_{k})\Big)d(\rho(T^{jm}(\mathbb{D}(\tau(\underline{a}_{1},\underline{a}'))),\rho(T^{jm}(\mathbb{D}(\tau(\underline{a}_{2},\underline{a}'))))	\end{equation*} for some $w_{k}\in \rho(T^{(u-j)m-k}(\mathbb{D}(\tau(\sigma^{jm}(\underline{a}_1),\underline{a}'))))$.  Then: 
			\begin{equation}\label{eq.5.7}
				d(I_{\underline{a}_{1}},I_{\underline{a}_{2}})\geq  \delta_1   \prod_{k=1}^{(u-j)m-1}\lambda(w_{k}).
			\end{equation}
			
			By (\ref{eq.5.6}) and (\ref{eq.5.7}), we have  	\begin{equation*} d(I_{\underline{a}_{1}},I_{\underline{a}_{2}})\geq  \Big(\prod_{k=1}^{(u-j)m-1} \dfrac{\lambda(w_{k})}{\lambda(t_{k})}\Big) \dfrac{\delta_1}{2} {  \operatorname{diam}(I_{\sigma^{jm}(\underline{a}_{1})})}  \geq \delta_1   \dfrac{ K_6^{-1}}{2} \operatorname{diam}(I_{\sigma^{jm}(\underline{a}_{1})}).	\end{equation*}
			
			The result follows taking $\delta_2= \frac{\delta_1 K_6^{-1}}{2}$. 
			

		\end{proof}

		\begin{lemma}\label{Lema 5.6}
			For each $\underline{a}\in I'_{u,m}(x')$ and $x\in C'(m)$, suppose that the closed subset $D_{x',\underline{ a}}$ of $\rho(\mathbb{D}(x'))$ satisfy:
			\begin{enumerate}
				\item {
					$D_{x',\underline{a}} = T_{\underline{a}(x')}(\mathbb{D}(\underline{a}(x')))$, $\forall \ \underline{a}\in I'_{m}(x') $;
				} 
				
				\item $D_{x',\underline{a},\underline{\tilde{a}}}\subset D_{x',\underline{a}}$, $\forall \ \underline{a}\underline{\tilde{a}}\in I'_{u+1,m}(x') $;
				\item
				For some  $t_{k}$ and $w_{k}$ in  $\rho(\tilde{T}^{um-k}(\mathbb{D}([\underline{a}\underline{a}']_{um}(x'))))$:
				\begin{equation}\label{diam2}
					\prod_{k=1}^{um}\lambda (t_{k})  \leq \frac{\operatorname{diam}(D_{x',\underline{a}})}{\operatorname{diam}(\rho(\mathbb{D}([\underline{a}\underline{a}']_{um}(x'))))}\leq  \prod_{k=0}^{um}\lambda  (w_{k});
				\end{equation}	
				
				\item There is $\delta>0$ such that for any $\underline{a}=(\underline{a_1},...,\underline{a_u}), \ \underline{b}=(\underline{b_1},...,\underline{b_u}) \in I'_{u,m}(x') $, with $u\geq 2$, if the greatest index $j$ such that $(\underline{a_1},...,\underline{a_j})=(\underline{b_1},...,\underline{b_j})$ is less than or equal to $u-2$, so $D_{x',\underline{a}}\cap D_{x',\underline{b}}=\emptyset$ and
				\begin{equation}\label{item3}
					d(D_{x',\underline{a}},D_{x',\underline{b}})\geq \delta \operatorname{diam}(D_{x',\pi_{jm}(\underline{a})}).
				\end{equation} 
			\end{enumerate}
			
			Denote	$\displaystyle C_{*}(x',m)=\bigcap_{u\geq1} \bigcup_{\underline{a}\in I'_{u,m}(x')} D_{x',\underline{a}},  $ then it is valid that  $$\dim_{H}  ( C_{*}(x',m)) \geq \textbf{d}^{**}(x',m).$$
		\end{lemma}
		
		\begin{proof}[Proof of the Lemma \ref{Lema 5.6}]
			For every $\underline{a}_{i}\in I'_{u,m}(x')$ we will consider closed subsets $J^{*}_{\underline{a}_{i} } \subset  \rho(\mathbb {D}(x'))$ such that $\operatorname{diam}(J^{*}_{\underline{a}_{i}}) = \operatorname{diam}(D_{x',\underline{ a}_{i}})$. 
			We will obtain the sets $J^{*}_{\underline{a}_{i}}$ by a translation of the $D_{x',\underline{a}_{i}}$ as follows.
			
			For every $u$, we define recursively the sets $J^{*}_{\underline{a}}$ by a translation of the sets $D_{x',\underline{a}_{i}}$.  In the step $u+ 1$,
			for each $\underline{a}=(\underline{a}_{1},..., \underline{a}_{u})\in I'_{u,m}(x')$, we define the set $J^{*}_{\underline{a}\underline{b}}:=J^{*}_{\underline{a}_{1}\underline{a}_{2}...{\underline{a}_{u}},\underline{b}}$ making a translation of the set $D_{x',\underline{a}\underline{b}}$ so that $J^{*}_{\underline{a}\underline{b}}\subset J^ {*}_{\underline{a}}$ and $J^{*}_{\underline{a}\underline{b}_{1}}\cap J^{*}_{\underline{a }\underline{b}_{2}}=\emptyset $. Since they are obtained from translations, we have that $J^{*}_{\underline{a}}$ is closed and $\operatorname{diam}(J^{*}_{\underline{a}})=
			\operatorname{diam}(D_{x',\underline{a}})$ for every $ \underline{a}\in I'_{u,m}(x')$.

			Consider: 	\begin{equation*}C^{*}(x',m)=\bigcap_{u\geq1} \bigcup_{\underline{a}\in I'_{u,m}(x')} J^{*}_{\underline{a}}.	\end{equation*}

			{Let us  prove that 
				$\dim_{H}  ( C^{*}(x',m)) \geq \textbf{d}^{**}(x',m)$.  Suppose by contradiction that  $\textbf{d}^{**}(x',m) > \alpha > \dim_{H}  ( C^{*}(x',m))$, then there exist open covers $\mathcal{U}$ of $C^*(x',m)$ with $H_\alpha(\mathcal{U})$ arbitrarilly small (what implies that $\diam \mathcal{U}$ is also arbitrarily small). We can suppose that  $\mathcal{U}$ is finite.}
			{		  Consider $\epsilon_0$ such that if $H_\alpha(\mathcal{U})< \epsilon_0$ then every $U \in \mathcal{U}$ intersects at most one element of $\mathcal{A}:= \{ J^*_{\underline{a}}, \underline{a} \in I'_m(x') \}$.}
			
			{Given $\underline{a} \in A^*_{\underline{a}',m}$, we can write $J^{*}_{\underline{a}} = \phi_{\underline{a}}(\mathbb{D})$, where we identify $\mathbb{D}$ with every $\mathbb{D}(x)$ and $ \phi_{\underline{a}}$ is the composition of $T_{\underline{a}(x)}$ with some translation.
			}
			{For every $\underline{a} \in I'_m(x')$, $\mathcal{U}_{\underline{a}} = \{ U \cap J^*_{\underline{a}}, U \in \mathcal{U} \}$ is an open cover of $J^*_{\underline{a}} \cap C_*(x',m)$, so $\phi_{\underline{a}}^{-1}(\mathcal{U}_{\underline{a}})$ is another open cover of $C^*(x',m)$ with }
			\begin{equation*}
				{H_\alpha(\phi_{\underline{a}}^{-1}(\mathcal{U}_{\underline{a}} )) \leq \sup_{x \in J^{*}_{\underline{a}}}  \|D\phi_{\underline{a}}^{-1}(x)\|^\alpha   H_\alpha(\mathcal{U}_{\underline{a}}) }
			\end{equation*}
			
			{Notice that by \eqref{diam}, \eqref{diam2} and the Chain rule, we have that  $\sup_{x \in J^{*}_{\underline{a}}}  \|D\phi_{\underline{a}}^{-1}(x)\| \leq K_6  \diam(D_{x',\underline{a}})$.}
			{It must follow that $H_\alpha(\phi_{\underline{a}_0}^{-1}(\mathcal{U}_{\underline{a}_0} )) < \epsilon_0$ for some $\underline{a}_0$, otherwise  }
			$${  H_\alpha(\mathcal{U})  \geq \sum_{\underline{a}\in I^{'}_{m}(x')} H_\alpha(\mathcal{U}_{\underline{a}} ) \geq \sum_{\underline{a}\in I^{'}_{m}(x')}  \sup_{x \in J^{*}_{\underline{a}}}  \|D\phi_{\underline{a}}^{-1}(x)\|^{-\alpha} H_\alpha(\phi_{\underline{a}}^{-1}(\mathcal{U}_{\underline{a}} ))
			}$$
			$${ \geq \epsilon_0 K_6^{-1} \sum_{\underline{a}\in I^{'}_{m}(x')} \diam(D_{x',\underline{a}'})^{\alpha} > \epsilon_0 
			}$$
			
			{Above we used that $\alpha < \mathbf{d}^*(x',m)$ to obtain the strict inequality. 	So $\tilde {\mathcal{U}} := H_\alpha(\phi_{\underline{a}_0}^{-1}(\mathcal{U}_{\underline{a}_0} ))$	is another open cover with fewer elements than $\mathcal{U}$  and $H_\alpha(\tilde{\mathcal{U}}) < \epsilon_0$. Repeating this process, it will give an open cover of $C^*(x',m)$ without elements, what is a contradiction. So we have proved that $\dim_{H}  ( C^{*}(x',m)) \geq \textbf{d}^{**}(x',m)$.
			}


			Now, let us define a surjective transformation $h^{*}: C_{*}(x',m)\to C^{*}(x',m)$ that is Lipschitz continuous. This will imply that  	\begin{equation*}\dim_{H}(C_{*}(x',m))\geq \dim_{H}(h^{*}(C_{*}(x',m)))=\dim_{H}(C^{*}(x',m))\geq \textbf{d}^{**}(x',m). 	\end{equation*}
			
			Given $c_{*}\in C_{*}(x',m)$, for every $u=1,2,...$ there is $\underline{a}\in I'_{u,m}(x')$ such that $c_{*}\in D_ {x',\underline{a}}$.  Define $h^{*}(c_{*})=\bigcap_{u=1}^{\infty}J^{*}_{\underline{a }}$, which is well defined because $\{J^{*}_{\underline{a}}\}$ is a decreasing  sequence of compact sets and $\lim_{u\to\infty}\operatorname{diam}(J^{*}_{\underline{a}})=0$.
			Moreover, $h^{*}$ is a bijection between $C_{*}(x',m)$ and $C^{*}(x',m)$.
			
			Given $ c_{*} \neq c^{*} $ in $ C_{*}(x',m) $, consider $ u_{0} $ the greatest index such that $h^{*} (c_{*})$ and $h^{*} (c^{*}) $ are in the same $J^{*}_{\underline{a}_{1}, \underline{a}_{2}, ..., \underline{a}_{u_{0}}}. $ So: 	\begin{equation*}\operatorname{dist}(h^{*}(c_{*}),h^{*}(c^{*}))\leq \operatorname{diam}(J^{*}_{\underline{a}_{1},\underline{a}_{2},...,\underline{a}_{u_{0}}}).	\end{equation*} 
			
			By \eqref{item3}, we have that 	\begin{equation*}d(c_{*},c^{*})\geq \delta \operatorname{diam}(D_{x',\pi_{u_{0}m}(\underline{a})})=\delta \operatorname{diam}(J^{*}_{\underline{a}_{1},\underline{a}_{2},...,\underline{a}_{u_{0}}}).	\end{equation*} Therefore
			$d(h^{*}(c_{*}),h^{*}(c^{*}))\leq \delta^{-1}d(c_{*},c^{*})$ what completes this proof.

		\end{proof}
		
		Now we can finish the proof of Theorem \ref{TeoB1}.
		
		\begin{proof}[Proof of Theorem \ref{TeoB1}]
			
			For any arbitrary $x'_{m} = \tau(\underline{a}'_{m}) \in C'(m)$, $\underline{a}'_{m} \in I'_{\infty, m}$,  
			%
			%
			the sets 	\begin{equation*}D_{\underline{a}}=\rho(T^{um}(\mathbb{D}(\tau(\underline{a},\underline{a}'_{m}))))	\end{equation*} satisfy Properties (1) and (2) above, and Proposition \ref{Prop 5.7} implies Property (3).
			
			Since $  \rho(\Delta'_{T}(x'_{m}))=\bigcap_{u=1}^{\infty} \bigcup_{\underline{a}\in I'_{u,m}(x'_m)} D_{\underline{a}},$    Lemma \ref{Lema 5.6} implies that 
			%
			\begin{equation*}\liminf_{m\to\infty} \dim_{H}(\rho(\Delta'_{T}(x'_{m}))) \geq  \liminf_{m\to\infty}  \textbf{d}^{**}(x,m)  = d_0.	\end{equation*}		
			
		\end{proof}

		\subsection{Unstable holonomy}\label{U.H}
		
		Consider an integer $m$ large so that  {Theorem \ref{TeoB1}} is valid.
		Given $x'\in C'(m)$ and any $x\in\mathbb{T}^{l},$ denote  $x_c= x_{c}(x', x)$ the middle point of the geodesic segment joining $x'$ to $x$. For each $t'\in \Delta'_{T}(x')$, take $t_{c} \in\Delta_{T}(x_{c})\cap W^{u}(t') $ such that $t_c=t_c(x_c,t')$ is the closest point to $t'$ in $W^{u}(t')$.  Define, for $t^{*}\in V$ and $L\geq 0$, the unstable set
		\begin{equation*} W^{u}_{L}(t^*):=\{t\in W^{u}(t^{*}): d(\pi(T^{-n}(t)),\pi(T^{-n}(t^{*})))\leq L, \ \forall  \ n\geq 0\}.	\end{equation*}

		We have that
		$t'\in W^{u}_{\frac{d}{2}}(t_{c})$  for  $\operatorname{d}=d(x',x)$,
		%
		%
		%
		%
		%
		%
		%
		%
		%
		therefore for each $t'\in\Delta'(x')$ we can associate a single $t\in W^{u}_{\frac{d}{2}}(t_{c})\cap\Delta(x)$. That is, we can define the following application:
		\begin{equation*}\tilde{h}:\Delta'_{T}(x') \longrightarrow  \Delta_{T}(x)\quad \mbox{with} \quad \tilde{h}(t')= t\in W^{u}_{\frac{d}{2}}(t_{c}) \cap\Delta_{T}(x).	\end{equation*}
		
		The map $\tilde{h}$ is the unstable holonomy from $\Delta'_{T}(x')$ to $\Delta_{T}(x)$. 
		By {Theorem \ref{TeoB1}}, we have that $\rho|_{\Delta'_{T}(t')}$ injective for each $t'\in C'(m)$.  Thus, the following application: 	\begin{equation*}h:=\rho \tilde{h}\rho^{-1}:\rho(\Delta'_{T}(x'))\to\rho(\Delta_{T}(x))	\end{equation*}
		is well defined. 
		%
		For  this $h$, we have the following result:
		\begin{theorem}\label{TeoB2}
			For every $x'\in C'(m)$, with $m $ large enough, there is a finite partition of $\rho(\Delta_{T}'(x'))$ in disjoint compact subsets $E_1, E_2, ..., E_{q}$ such that $h|_{E_{i}}, \ i=1,2,...,q$, is injective and has a continuous Lipschitz inverse.
		\end{theorem}
		
		\begin{proof}
			Consider $m, \ x'\in C'(m), x\in\mathbb{T}^{l}$, $d$, $\tilde h$ and $h$  as before. By Proposition \ref{Prop 5.6},  for any $x^{*}\in C'(m)$ and $t_1,t_2\in\Delta'_{T}(x^{*})$, with $\pi(T^{-m}(t_1))\neq\pi(T^{-m}(t_2))$, it is valid that 	\begin{equation*}d(\rho(W^{u}_{\delta_1}(t_{1})), \rho(W^{u}_{\delta_1}(t_{2})) )\geq \delta_1,	\end{equation*}
			because $W^{u}_{\delta_1}(t_{1})\subset T^{m}\big( \Delta'_{T}(\pi(T^{-m}(t_1)))\big) $ and $W^{u}_{\delta_1}(t_{2})\subset T^{m}(\Delta'_{T}(\pi(T^{-m}(t_2))))$. 
			
			Consider $\mathbf{B} = B(x,d/2)$ and $k_1$ positive integer large enough such that the diameter of every component of $\varphi^{-k_{1}m}(\mathbf{B})$ is at most $\delta_1$.
			
			Suppose that  $\{x_{1},x_{2},...,x_{q}\} =  \varphi^{-k_{1}m}(x')\cap C'(m)$.  Define  	\begin{equation*}E_{j}=\rho(T^{k_{1}m}(\Delta_{T}'(x_j))), \  \  j=1,...,q.	\end{equation*}
			We have that $\rho(\Delta_{T}'(x'))= \bigcup_{j=1}^{q}E_{j}$. Since  $\rho|_{\Delta_{T}'(x^{*})}$ is injective for every $x^{*}\in C'(m)$, it is valid that the sets $E_{j}, \ j=1,...,q$ are pairwise disjoint.
			
			Let us prove now that the restriction $h|_{E_{j}}, \ j=1,...,q$, is injective and that its inverse is Lipschitz continuous. Consider $t',  e'\in \rho(\Delta_{T}'(x'))$ so that $t',e'$ is in the same $E_{j}$, for some $j=1,2,...,q$,  with $t'\neq e'$.  There exist unique $\tilde{t'}$ and $\tilde{e'}$ in $\Delta'_{T}(x')$ with $\rho(\tilde{t'})=t'$ and $\rho(\tilde{e'})=e'$.  Denote: 	\begin{equation*}\tilde{t}=\tilde{h}(\tilde{t'})  , \quad t=\rho(\tilde{t})=\rho(\tilde{h}(\tilde{t'}))=h(t') , 	\end{equation*}
			\begin{equation*}\tilde{t'}_{i}=T^{-im}(\tilde{t'})  , \quad  \tilde{t}_{i}=T^{-im}(\tilde{t}) , \quad \pi(\tilde{t'}_{i})=\pi(\tilde{e'}_{i})=x'_{i},  	\end{equation*}
			\begin{equation*}\tilde{e}=\tilde{h}(\tilde{e'}), \quad e=\rho(\tilde{e})=\rho(\tilde{h}(\tilde{e'}))=h(e'), 	\end{equation*}
			\begin{equation*}\tilde{e'}_{i}=T^{-im}(\tilde{e'}) , \quad \tilde{e}_{i}=T^{-im}(\tilde{e}) , \quad  \pi(\tilde{t}_{i})=\pi(\tilde{e}_{i})=x_{i}	\end{equation*}
			for $i=0,1,2...$
			
			Let $k_0$ be the greatest integer such that  $\pi(T^{-k_{0}m}(\tilde{t'}))=\pi(T^{-k_{0}m}(\tilde{e'}))$.  Since $t',e'$ are in the same $E_{j}$, it follows that $k_0\geq k_1$.  Then
			\begin{equation*}d(\pi(T^{-k_{0}m}(\tilde{t'})),\pi(T^{-k_{0}m}(\tilde{e}) )=d(x'_{k_{0}},x_{k_{0}})\leq \delta_1,	\end{equation*} that is, $\tilde{t_{k_{0}}}\in W^{u}_{\delta_{1}}(\tilde{t'}_{k_{0}})$ and $\tilde{e_{k_{0}}}\in W^{u}_{\delta_{1}}(\tilde{e'}_{k_{0}})$. 
			Then we conclude that  $d(\rho(\tilde{t_{k_{0}}}),\rho(\tilde{e_{k_{0}}}))\geq \delta_1,$ that is,
			$d(\rho(T^{-k_{0}m}(\tilde{t})),\rho(T^{-k_{0}m}(\tilde{e})))\geq \delta_1$.  
			
			Therefore,
			\begin{equation*}d(t,e)=d(\rho(\tilde{t}),\rho(\tilde{e}))\geq\big( \prod_{w=1}^{k_{0}m}\lambda (t_{w})\big) d(\rho(T^{-k_{0}m}(\tilde{t})),\rho(T^{-k_{0}m}(\tilde{e})))\geq \delta_1 \big( \prod_{w=1}^{k_{0}m}\lambda  (t_{w})\big),	\end{equation*}
			for some $t_{w}=(x_{w},z_{w})\in \rho(T^{w}(\mathbb{D}(x)))$. Then, $h(t')\neq h(e') $, what means that $h|_{E_{j}}$ is injective.
			
			On the other hand, we have that $d(\rho(T^{-k_{0}m}(\tilde{t'})),\rho(T^{-k_{0}m}(\tilde{e'})))\leq 2$. Then 	\begin{equation*}d(t',e')=d(\rho(\tilde{t'}), \rho(\tilde{e'}))\leq\Big( \prod_{w=1}^{k_{0}m}\lambda(t^{*}_{w})\Big)d(\rho(T^{-k_{0}m}(\tilde{t'})),\rho(T^{-k_{0}m}(\tilde{e'})))\leq 2\big( \prod_{w=1}^{k_{0}m}\lambda(t^{*}_{w})\big),	\end{equation*}
			where $t^{*}_{w}=(x^{*}_{w},z^{*}_{w})\in \rho(T^{w}(\mathbb{D}(x')))$. We conclude that 
			\begin{equation}\label{eq.5.16}
				d(t',e')\leq \frac{2}{\delta_1} \frac{\prod_{w=1}^{k_{0}m}\lambda(t^{*}_{w})}{\prod_{w=1}^{k_{0}m}\lambda(t_{w})}   d(t,e).
			\end{equation} 	
			
			Notice that $d(x_{w},x^{*}_{w})\leq \underline{\beta}^{-w}$ and $d(z_{w},z^{*}_{w})\leq K_7\overline\lambda^{w} $ for some $K_7>0$.  Then 	\begin{equation*}d(t_{w},t^{*}_{w})\leq \underline{\beta}^{-w}+ K_7\overline\lambda^{w}.  	\end{equation*}
			So 	\begin{equation*}\lambda_{2}(t^{*}_{w})\leq \overline\lambda d(t_{w},t^{*}_{w})+\lambda(t_{w})   \leq \overline\lambda(\underline{\beta}^{-w}+ K_7 \overline\lambda_{max}^{w})+\lambda(t_{w}). 	\end{equation*}
			
			Using it in (\ref{eq.5.16}), we have  	
			\begin{equation*}
				d(t',e')\leq 
				\frac{2}{\delta_1}\prod_{w=1}^{\infty}\Big( \overline\lambda (\underline{\beta}^{-w}+ K_7\overline\lambda^{w})\underline\lambda^{-1}+1 \Big) d(t,e) \leq K_8 d(t,e) 	\end{equation*}
			for 
			$K_8 = \prod_{w=1}^{\infty}\big( \overline\lambda (\underline{\beta}^{-w}+ K_7\overline\lambda^{w}) \underline\lambda^{-1}+1 \big)$.  
			%

		\end{proof}

		\subsection{{Proof of the Theorems}} 
		
		\begin{proof}[Proof of the Theorem A]
			For every $\epsilon>0$ we can consider  $m$ large,
			$C'(m)\subset \mathbb{T}^{l} $, $\Delta'_{T}(m)\subset\pi^{ -1}(C'(m))\cap \Delta_{T}$ and $\Delta'_{T}(x')$ for some  $x'\in C'(m)$, 
			such  that 
			$$\dim_H(\rho (\Delta'_T(x')))  \geq d_0 - \epsilon	$$
			for every $x' \in C' = C'(m)$.
			For every $x\in \mathbb{T}^{l}$, consider the applications 
			$\tilde{h}:\Delta'_{T}(x')\longrightarrow  \Delta_{T}(x)$  and  $ h:=\rho \tilde{h}\rho^{-1}:\rho(\Delta'_{T}(x'))\to\rho(\Delta_{T}(x)).	$

			Partitioning $\rho( \Delta'_T(x'))=E_1 \cup \cdots \cup E_r$ as in Theorem \ref{TeoB2}, we have $\dim_H E_i \geq d_0 - \epsilon$ for some $i$, then $\dim_H \rho(\Delta_T(x)) \geq \dim h(E_i) \geq d_0 - \epsilon$.  
			Since  $\rho$ is  Lipschtiz continuous, we have 	$$d_0 - \epsilon \leq \dim_{H}\rho(\Delta_{T}(x))\leq \dim_{H}( \Delta_{T }(x)) \leq d_0	$$
			for every $x$ and $\epsilon$. 
			Therefore,
			$\dim_{H}(\Delta_{T}(x))=d_{0}$ for every $x \in \mathbb{T}^l$,
			what implies of the Theorem \ref{Teo 5.1}.

		\end{proof}

		For the proof of Corollary B, we  use the following result:
		
		\begin{theorem}\label{Teo5.1}
			Let $F$ be a measurable subset of $\mathbb{R}^n$, $E$ a measurable subset of $\mathbb{R}^{k}$, with $k < n$  and $m_{k}(E)>0$, $L$ a $(n-k)$-dimensional subspace of $\mathbb{R}^n$ and $L_{x}=L+ x$ the translation of $L$ by  $x\in E$. If
			$\dim_{H}(F \cap L_x)\geq t > 0$ for Lebesgue almost every $x\in E$, then $\dim_H (F) \geq t + \dim_H (E)$.
		\end{theorem}
		\begin{proof}
			Denote by $m_{d}$ the $d-$dimensional Lesbegue measure. Let us first prove that if $k\leq s\leq n$, then \begin{equation*}\int_{\mathbb{R}^{k}}\mathcal{H}^{s-k}(F\cap L_{x})\operatorname{d}x\leq \mathcal{H}^{s}(F).\end{equation*} 
			
			Given $\epsilon>0$, consider $\{U_{i}\}_{i}$ a $\delta-$cover of $F$ such that $\sum_{i}|U_{i} |^{s}\leq \mathcal{H}^{s}_{\delta}(F)+\epsilon.$ Each $U_{i}$ is contained in a $n-$cube $S_{i} $ of edge $|U_{i}|$ and with faces parallel to the canonical axes of $\mathbb{R}^{n}.$ Let $ \mathbbm{1}_{S_{i}}$ be the indicator function of $S_{i}$, that is, $\mathbbm{1}_{S_{i}}(y)=1$, if $y\in S_{i}$ and $\mathbbm{1}_{S_{i}}(y)=0$ if $y\notin S_{i}.$ For each $x\in \mathbb{R }^{k}$, the sets $S_{i}\cap L_{x}$ form a $\sqrt{(n-k)}\delta-$cover of $F\cap L_{x}.$ So, for $\delta'= \sqrt{(n-k)}\delta$, we have 
			
			\begin{align*}
				\mathcal{H}^{s-k}_{\delta'}(F\cap L_{x}) & \leq \displaystyle\sum_{i}|S_{i}\cap L_{x}|^{s-k}
				=\displaystyle\sum_{i}|S_{i}\cap L_{x}|^{s-n}|S_{i}\cap L_{x}|^{n-k}\\
				&  \leq \displaystyle\sum_{i}(\sqrt{n-k})^{s-n}|U_{i}|^{s-n}\displaystyle\int_{L_{x}}\mathbbm{1}_{S_{i}}\operatorname{d}m_{n-k}(y) .
			\end{align*}
			
			So, \begin{equation*}
				\int \mathcal{H}^{s-k}_{\delta'}(F\cap L_{x})\operatorname{d}m_{k}(x)\leq (n-k)^{\left(\frac{s-n}{2} \right) }\displaystyle\sum_{i}|U_{i}|^{s-n} \displaystyle\int \left(\displaystyle\int_{L_{x}}\mathbbm{1}_{S_{i}}\operatorname{d}m_{n-k}(y) \right)\operatorname{d}m_{k}(x). 
			\end{equation*} Note that $(n-k)^{\left(\frac{s-n}{2} \right)}\leq 1$, since $n-k\geq 1$ and $s-n\leq 0.$ Furthermore, \begin{equation*}\int \left(\int_{L_{x}}\mathbbm{1}_{S_{i}}\operatorname{d}m_{n-k}(y) \right)\operatorname{d}m_{k}(x)=\operatorname{vol}_{n}(S_{i}\cap L_{x})\leq |U_{i}|^{n}.\end{equation*} So, \begin{equation*}\int \mathcal{H}^{s-k}_{\delta'}(F\cap L_{x})\operatorname{d}m_{k}(x)\leq\displaystyle\sum_ {i}|U_{i}|^{s}\leq \mathcal{H}^{s}_{\delta}(F)+\epsilon.\end{equation*}
			
			Since $\epsilon>0$ is arbitrary and $\mathcal{H}_{\delta}^{s}$ is monotonic in $\delta$, using the Monotone Convergence Theorem, taking $\delta\to 0$, we have that
			\begin{equation}\label{eq hausd.}
				\displaystyle\int_{\mathbb{R}^{k}}\mathcal{H}^{s-k}(F\cap L_{x})\operatorname{d}x\leq \mathcal{H}^{s}(F).
			\end{equation}

			Now, let us consider $s=\dim_{H}(F)+\epsilon'$, in particular $\mathcal{H}^{s}(F)=0.$
			Note that we must have $dim_{H}(F)>k$. In fact, by hypothesis, there exists a set $E'$ of points $x$ in $E$ with positive Lebesgue measure such that $\dim_{H}(F\cap L_{x})\ge t>0$ for all $x\in E'$. Given $\epsilon''>0$, for each $m$ consider the sets 
			$E_{m}=\{x\in E':\mathcal{H}^{t-\epsilon''}(F\cap L_{x})> \frac{1}{m}\}$. We have $E_{m}\subset E_{m+1}$ and $E'\subset\displaystyle\bigcup_{m}E_{m}.$ Then, $\displaystyle\lim_{m}m_{k}(E_ {m})\geq m_{k}(E')>0$ and so, there exists $N$ such that $m_{k}(E_{N})>0$. We have  \begin{equation*}
				\int_{\mathbb{R}^{k}}\mathcal{H}^{t-\epsilon''}(F\cap L_{x})\operatorname{d}x\geq\displaystyle\int_{E_{N}}\mathcal{H}^{t-\epsilon''}(F\cap L_{x})\operatorname{d}x\geq\displaystyle\int_{E_{N}}\dfrac{1}{N}\operatorname{d}x=\dfrac{1}{N}m_{k}(E_{N}) >0
			\end{equation*} and, by  (\ref{eq hausd.}),  \begin{equation*}\int_{\mathbb{R}^{k}}\mathcal{H}^{t-\epsilon''}(F\cap L_{x})\operatorname{d}x\leq \mathcal{H}^{t-\epsilon''+k}(F),\end{equation*} it follows that $\mathcal{H}^{t-\epsilon''+k}(F)>0$, that is, $\dim_{H}(F)\geq t-\epsilon''+k$ for every $\epsilon''>0$ small.
			{Since $k \geq \dim_H(E)$, it follows what we want.}
		\end{proof}
		
		\begin{proof}[Proof of Corollary B]	
			Consider $x_{0}  \in   \mathbb{T}^{l} $ and a ball $\mathbf{B}(x_{0},r)\subset\mathbb{T}^{l}$, with $r$ small, also consider the identification of the torus $\mathbb{T}^{l}$ with $E=[0,1]^{l}\subset\mathbb{R}^{l}$. The attractor $\Delta_{T}$ is identified with a subset of $[0,1]^{l}\times \mathbb{R}^{d}\times\mathbb{R}^{p} \subset  \mathbb{R}^{l+d+p}.$
			
			Let us apply Theorem \ref{Teo5.1} for 
			$n=l+d+p$, $k=l$, $F=\Delta_{T}$, $E=\mathbf{B}(x_{0},r) $ and  $L_{x}=\{x\}\times\mathbb{R}^{d}\times\mathbb{R}^{p}=(x+\{ 0\})\times\mathbb{R}^{d}\times\mathbb{R}^{p} $ for each $x\in E$ ($L_X$ is an affine subspace of dimension $d+p=n-k$). We have that $F\cap L_{x}=\Delta_{T}(x)$ and therefore \begin{equation*}
				\dim_{H}(F\cap L_{x})=\dim_{H}(\Delta_{T}(x))\geq d_{0} \quad  \forall \  x\in E.
			\end{equation*} Thus, using Theorem \ref{Teo5.1}, it follows that \begin{equation*}
				\dim_{H}(\Delta_{T})=\dim_{H}(F)\geq d_{0}+\dim_{H}(E)=d_{0}+\dim_{H}(\mathbf{B}(x_{0},r))= l+d_{0}.
			\end{equation*} 
			This finishes the proof.

		\end{proof}
		
		Finally, let us notice the continuity of $d_0$.
		
		\begin{proof}[Proof of Corollary C]{
				Let us fix some $T_0=T(\varphi_0, \nu_0, \psi_0)$ and its codification as described in Section 2.}
			{ By structural stability, we can consider the same codification (same shift corresponding to a Markov partition with the same letters) for every $\varphi$ in a neighborhood of $\varphi_0$ and this induces the same codification in a neighborhood $\mathcal{V}$ of $T_0$ with respect to the $C^r$-topology.}
			{		Then we can consider the same shift $\sigma:\Sigma_A\to \Sigma_A$ so that  $\sigma^{-1}$ is conjugated to every $T \in \mathcal{V} $ by the function $$ h_T(\underline{\tilde{a}})=(\tau_T \circ R (\underline{\tilde{a}}^-) , S_T(\tau_T \circ R (\underline{\tilde{a}}^-) ,\underline{a})), $$ where $\underline{\tilde{a}}=(\underline{\tilde{a}}^-,\underline{{a}})$ and  we write $S_T$, $h_T$ and $\tau_T$ to emphasize their dependence on $T$.}
			
			{ We have that $\Phi: \mathcal{V} \to S_T \in C^0( \tilde M, \R)$ is continuous since it is the uniform limit of a family of contractions (here we endow $\mathcal{U}$ with the $C^r$-topology and $C^0$ with the supremum norm).} 
			
			{Denoting by $h_{\varphi_0,\varphi}$  the conjugation between the expanding maps $\varphi_0$ and $\varphi$, it is well known that it is continuous with respect to $\varphi$, so $\tau_T= h_{\varphi_0,\varphi}  \circ \tau_{T_0}$ is continuous with respect to $T$. Then $T \in \mathcal{V} \mapsto h_T \in C^0(\Sigma_A, \mathbb{T}^l \times \R^p)$ is continuous}

			{This imply that the potential   function $T \in \mathcal{V} \mapsto \phi_T \in C^0(\Sigma_A, \R)  $ is continuous, for $ 	\phi_T(\underline{\tilde{a}})=\log\lVert D_{y}\nu(h_T(\underline{\tilde{a}}))|_{Y }\rVert$. Since the pressure is continuous with respect to the potential, it follows that the pressure function $B(d, T) := P(\sigma, d \phi_T)$ is also continuous.}

			{Now we can prove that $T \mapsto d_0(T)$ is continuous. Actually, since for every $T$ there is only one $d=d_0(T)$ such that $P(d, \phi(T)) =0$, if $T_n$ converges to $T$ in the $C^r$-topology then $d_0(T_n) \in [0,p]$ have as accumulation point the unique  $d$ such that $ P(d, \phi(T)) = \lim_{n_k} P(d_0(T_{n_k}), \phi(T_{n_k}))=0$, so $d=d_0(T)$.
			}
			
			{ When $p=1$ the set $\mathcal{U}$ is $C^1$-open because every $\nu: \mathbb{T}^l \times \R \to \R$ is conformal (condition \eqref{eq 2.1.1} is always valid for $p=1$), and the other conditions are always $C^1$-open.
			}
		\end{proof}

		
\end{document}